\documentclass[12pt,reqno,sumlimits,oneside]{amsart}

\usepackage[T1]{fontenc}
\usepackage[utf8]{inputenc}
\usepackage{lmodern}

\usepackage{babel}

\usepackage[a4paper,margin=2.4cm]{geometry}

\usepackage{xifthen}

\usepackage{mathrsfs,dsfont}
\usepackage[mathscr]{euscript}

\usepackage{xargs}

\usepackage{graphicx}
\usepackage[disable,table,svgnames,x11names]{xcolor}

\usepackage[colorinlistoftodos,prependcaption,textsize=tiny,textwidth=3.5cm]{todonotes}
\presetkeys{todonotes}{fancyline}{}
\newcommandx{\todoin}[2][1=]{\todo[inline, caption={todo}, #1]{%
    \begin{minipage}{\textwidth-20pt}#2\end{minipage}}}

\newcommandx{\richard}[2][1=]{\todo[linecolor=blue,backgroundcolor=blue!25,bordercolor=blue,#1]{#2}}
\newcommandx{\richardin}[2][1=]{\richard[inline, caption={Richard}, #1]{%
    \begin{minipage}{\textwidth-30pt}#2\end{minipage}}}
\newcommandx{\tomasz}[2][1=]{\todo[linecolor=red,backgroundcolor=red!25,bordercolor=red,#1]{#2}}
\newcommandx{\tomaszin}[2][1=]{\tomasz[inline, caption={Tomasz}, #1]{%
    \begin{minipage}{\textwidth-25pt}#2\end{minipage}}}

\usepackage{amssymb}
\usepackage{amsthm}

\usepackage[all]{xy}
\newcommand{\vcxymatrix}[1]{\vcenter{\xymatrix{#1}}}

\usepackage{enumerate}

\newcounter{proof}
\newenvironment{myproof}%
{\stepcounter{proof}\begin{proof}}%
{\end{proof}}%
\newcounter{proofstep}[proof]
\newenvironment{proofstep}[1][]%
{\refstepcounter{proofstep}\bigskip\par\noindent%
  \ifthenelse{\isempty{#1}}
  {\textsf{Step \theproofstep. }}
  {\textsf{#1.}}
  \noindent}%
{\par}%
\newcounter{proofcase}[proof]%
\newenvironment{proofcase}[1][]%
{\refstepcounter{proofcase}\bigskip\par\noindent%
  \ifthenelse{\isempty{#1}}
    {\textsf{Case \theproofcase. }}
    {\textsf{#1.}}
  \noindent}%
{\par}%

\newcounter{smallromans}%
\newenvironment{romanenumerate}%
{\begin{list}{{\normalfont\textrm{(\roman{smallromans})}}}
    {\usecounter{smallromans}\setlength{\itemindent}{0cm}
      \setlength{\leftmargin}{5.5ex}\setlength{\labelwidth}{5.5ex}
      \setlength{\topsep}{.5ex}\setlength{\partopsep}{.5ex}
      \setlength{\itemsep}{0.1ex}}}
  {\end{list}}%

\newcounter{smallromansdash}%
{\begin{list}{{\normalfont\textrm{(\roman{smallromansdash}$'$)}}}%
    {\usecounter{smallromansdash}\setlength{\itemindent}{0cm}%
      \setlength{\leftmargin}{5.5ex}\setlength{\labelwidth}{5.5ex}%
      \setlength{\topsep}{.5ex}\setlength{\partopsep}{.5ex}%
      \setlength{\itemsep}{0.1ex}}}%
  {\end{list}}%

\newcounter{bigromans}%
{\begin{list}{{\normalfont\textrm{(\Roman{bigromans})}}}%
    {\usecounter{bigromans}\setlength{\itemindent}{0cm}%
      \setlength{\leftmargin}{5.5ex}\setlength{\labelwidth}{6ex}%
      \setlength{\topsep}{.5ex}\setlength{\partopsep}{.5ex}%
      \setlength{\itemsep}{0.1ex}}}%
  {\end{list}}%

\usepackage{varioref}
\usepackage{hyperref}
\hypersetup{
  colorlinks,
  allcolors=DarkBlue
}
\usepackage[nameinlink]{cleveref}
\crefname{thm}{theorem}{theorem}

\newcounter{maintheorem}

\newtheorem{mainth}[maintheorem]{Theorem}

\newtheorem{thm}{Theorem}[section]
\newtheorem*{thm*}{Theorem}

\newtheorem{cor}[thm]{Corollary}
\newtheorem{lem}[thm]{Lemma}

\theoremstyle{definition}
\newtheorem{dfn}[thm]{Definition}
\theoremstyle{remark}

\newtheorem{rem}[thm]{Remark}

\numberwithin{equation}{section}

\DeclareMathOperator{\sign}{sign}

\DeclareMathOperator{\supp}{supp}
\DeclareMathOperator{\dist}{dist}

\DeclareMathOperator*{\esssup}{ess\,sup}

\DeclareMathOperator{\spn}{span}

\DeclareMathOperator{\dtree}{2^{<\omega}}

\DeclareMathOperator{\cond}{\mathsf{E}}

\newcommand{\bmo}{\ensuremath{\mathrm{BMO}}}

\newcommand{\cof}{{\rm{cof}}}

\begin{document}

\title[Stopping-time Banach spaces]%
{Factorisation in stopping-time Banach spaces: identifying unique maximal ideals}%

\author[T.~Kania]{Tomasz Kania}%
\address{Tomasz Kania, Mathematical Institute\\Czech Academy of Sciences\\\v Zitn\'a 25 \\115 67
  Praha
  1\\Czech Republic and Institute of Mathematics and Computer Science\\ Jagiellonian University\\
  {\L}ojasiewicza 6, 30-348 Krak\'{o}w, Poland}%
\email{kania@math.cas.cz, tomasz.marcin.kania@gmail.com}%

\author[R.~Lechner]{Richard Lechner}%
\address{Richard Lechner, Institute of Analysis, Johannes Kepler University Linz, Altenberger
  Strasse 69, A-4040 Linz, Austria}%
\email{richard.lechner@jku.at}%

\date{\today}



\thanks{The first-named author acknowledges with thanks funding received from SONATA 15
  2019/35/D/ST1/01734.  The second-named author was supported by the Austrian Science Foundation
  (FWF) Pr.Nr.\@ P32728.}

\begin{abstract}
  Stopping-time Banach spaces is a collective term for the class of spaces of event\-ually null
  integrable processes that are defined in terms of the behaviour of the stopping times with respect
  to some fixed filtration.  From the point of view of Banach space theory, these spaces in many
  regards resemble the classical spaces such as $L^1$ or $C(\Delta)$, but unlike these, they do have
  unconditional bases.\smallskip

  In the present paper, we study the canonical bases in the stopping-time spaces in relation to
  factorising the identity operator thereon.  Since we work exclusively with the dyadic-tree
  filtration, this setup enables us to work with tree-indexed bases rather than directly with
  stochastic processes.  \emph{En route} to the factorisation results, we develop general criteria
  that allow one to deduce the uniqueness of the maximal ideal in the algebra of operators on a
  Banach space.  These criteria are applicable to many classical Banach spaces such as (mixed-norm)
  $L^p$-spaces, \bmo, $\mathrm{SL^\infty}$, and others.
\end{abstract}

\maketitle%

\makeatletter%
\providecommand\@dotsep{5}%
\def\listtodoname{List of Todos}%
\def\listoftodos{\@starttoc{tdo}\listtodoname}%
\makeatother%

\section{Introduction and main results}
  
The familiar Banach spaces $L^1[0,1]$ and $C[0,1]$ are fundamental examples of classical Banach
spaces, which do not have an unconditional Schauder basis; actually neither space embeds into a
space with such a basis.  Rosenthal introduced in an unpublished manuscript an analogue of
$L^1[0,1]$, denoted $S^1$, that admits an unconditional Schauder basis. (Strictly speaking, $S^1$ is
an analogue of $L^1(\Delta)$, the space of integrable functions on the Cantor group $\Delta$, \emph{i.e.}, the
product $\{-1,1\}^{\mathbb N}$, with respect to the normalised Haar measure, which is the product of
infinitely many copies of the coin-toss measure on $\{-1,1\}$; this space is isometric to
$L^1[0,1]$, though.)  The space $S^1$ naturally comes in tandem with a space denoted by $B$, which
is a~space with an unconditional Schauder basis that resembles in many ways the space $C(\Delta)$ of
continuous functions on $\Delta$; the analogies between $S^1$ and $B$ and their classical
counterparts go deeper and will be delineated in subsequent paragraphs.\smallskip

Buehler in her Ph.D.~thesis (\cite[Section~3.1]{MR2691568}) considered the space $S^2$ that may be
viewed as a~certain convexification of $S^1$ and proved that $\ell^p$ does not embed therein for
$p\in [1,2)$.  The space $S^2$ is a member of a broader scale of spaces $S^p$ parametrised by
$p\in [1,\infty)$, whose left end-point is $S^1$ indeed.  Schechtman proved in an unpublished
manuscript that $S^1$ contains isometric copies of $\ell^p$ for all $p\in [1,\infty)$ and more
generally, for every $p\in [1,\infty)$, $S^p$ contains isometric copies of $\ell^q$ for
$q\geqslant p$.  The space $S^1$ was studied further by Dew in his Ph.D.~thesis~\cite{dew:2003} who
proved that $S^1$ contains isomorphic copies Orlicz sequence spaces $\ell^M$ for a rather wide class
of Orlicz functions $M$.\smallskip

The reader is owed an explanation of how the spaces $S^p$ and $B$ are constructed and why we call
them \emph{stopping-time Banach spaces}.  The following construction to an extent follows
\cite{dew:2003} (who considered only the case $p=1$), however alternative description based on
martingale differences may be found in~\cite{bangodell:89}.

\smallskip

Let $(\Omega, \mathcal{F}, \mathsf P)$ be a probability space and let
$\mathbb F = (\mathcal{F}_n)_{n=0}^\infty$ be a fixed filtration in $\mathcal{F}$.  A~stochastic
process $(X_n)_{n=0}^\infty$ on $(\Omega, \mathcal{F}, \mathsf P)$ is $\mathbb F$-\emph{adapted},
whenever $X_n$ is $\mathcal{F}_n$ measurable ($n\in \mathbb N$).  A~\emph{stopping time} is an
$\mathbb{N}\cup \{\infty\}$-valued random variable $T$ on $(\Omega, \mathcal{F}, \mathsf P)$ such
that for each $n$, we have $[T=n]\in \mathcal{F}_n$.  Let $\mathcal{T}$ denote the family of all
stopping times on $(\Omega, \mathcal{F}, \mathsf P)$ with respect to $\mathbb F$.  The space
$S^p_{\mathbb{F}}$ ($p\in [1,\infty)$) is the completion of the space of all eventually null
$p$-integrable processes $X=(X_n)_{n=0}^\infty$ with respect to the norm
\begin{equation}\label{formula:prob}
  \|X\|_{S^p_{\mathbb F}} = (\sup_{T\in \mathcal{T}} \mathsf E |X_T|^p)^{1/p}.
\end{equation}
In this paper, we shall be exclusively interested in the `dyadic-tree filtration' in Borel
$\sigma$-algebra of $[0,1]$.  Let us denote by $\mathcal{F}_n$ the $\sigma$-algebra generated by the
dyadic intervals $\{ [\tfrac{j-1}{2^n}, \tfrac{j}{2^n}]\colon 1\leqslant j \leqslant 2^n\}$ (in
particular, each algebra $\mathcal{F}_n$ is finite so $\mathcal{F}_n$-measurable random variables
assume only finitely many values almost surely) and put
$\mathbb F_d = (\mathcal{F}_n)_{n=0}^\infty$.  Among the stopping-time spaces, in the present paper,
we shall be thus concerned with the spaces $S^p_{\mathbb F_d} = S^p$ for $p\in [1,\infty)$
together with the space $B$ we are about to define.\smallskip

Let $2^{< \omega}$ denote the binary tree, which indexes the normalised Haar basis
$(h_t)_{t\in 2^{< \omega}}$ of $C(\Delta)$ (for details, see~\cite[Definition~2.2.2]{Semadeni:82}).
The space $B$ may be viewed as the space with a~minimal 1-unconditional Schauder basis that
dominates the Haar basis in $C(\Delta)$.  More specifically, $B$ is the completion of the space
$c_{00}(\dtree )$ (or the linear span of $ \{h_t\colon t\in \dtree \}$ in $C(\Delta)$) with respect
to the norm
\begin{equation*}
  \left\| (a_t)_{t\in \dtree} \right\|_B
  = \sup\Bigl\{\big\|\sum_{t\in \dtree} a_t\varepsilon_t h_t\big\|_{C(\Delta)}\colon
  t\in \dtree , |\varepsilon _t| = 1
  \Bigr\}\quad \big((a_t)_{t\in \dtree}\in c_{00}(\dtree )\big).
\end{equation*}
As observed in~\cite[Proposition~1]{bangodell:89}, the norm in $B$ may be isometrically realised as
\begin{equation*}
  \left\| (a_t)_{t\in \dtree} \right\|_B
  = \sup\Bigl\{\sum_{t\in \mathcal{A}} |a_t|\colon
  \mathcal{A}\text{ is a branch of }\dtree   \Bigr\}\quad \big((a_t)_{t\in \dtree}\in c_{00}(\dtree )\big).
\end{equation*}

In a similar fashion, the norm in $S^p$ ($p\in [1,\infty)$) can be seen as arising from the
completion of
\begin{equation*}
  \|((a_t)_{t\in \dtree}\|_{S^p}
  = \sup\Big\{ (\sum_{s\in\mathcal{A}} |a_s|^p)^{1/p}\colon
  \mathcal{A}\subset \dtree \text{ is an antichain}
  \Big\}
  \quad \big((a_t)_{t\in \dtree}\in c_{00}(\dtree )\big).
\end{equation*}

The standard unit vector basis $(e_t)_{t\in \dtree}$ of $c_{00}(\dtree)$ is then a 1-unconditional Schauder basis of $S^p$.  Furthermore, if $(e_t^*)_{t\in \dtree}$ denotes the coordinate functionals
associated to $(e_t)_{t\in \dtree}$, then $B$ may be viewed as a closed subspace of $D = (S^1)^*$
spanned by $(e_t^*)_{t\in \dtree}$ (\cite[Proposition 1]{bangodell:89}).

In the present paper we extend the scale $S^p$ defined in terms of $\ell^p$-spaces
($p\in [1,\infty)$) to spaces, that we denote by $S^E$, parametrised by spaces having a
$1$-subsymmetric Schauder basis (with respect to the standard linear order of the dyadic tree).
Even though the definition is not directly probabilistic, we still call $S^E$ stopping-time Banach
spaces.  (It is still possible to define these spaces in the spirit of \eqref{formula:prob}, however
we shall not require it here.)

\smallskip

Let $E$ be space a $1$-subsymmetric Schauder basis $(e_t)_{t\in \dtree}$. For a subset $\mathcal{A}\subseteq \dtree $, let $P_{\mathcal{A}}$ denote the standard basis
projection on $\mathcal{A}$, which is well defined by unconditionality of $(e_t)_{t\in \dtree}$.  We
define the spaces $S^E$ and $B^E$ as completions of $c_{00}$ with respect to the norms of
$(a_t)_{t\in \dtree}\in c_{00}(\dtree )$:
\begin{equation*}
  \begin{array}{lcl}
    \|(a_t)_{t\in \dtree}\|_{S^E}
    &=& \sup\Big\{
        \| P_{\mathcal{A}} (a_t)_{t\in \dtree}\|_E \colon \mathcal{A}\subset \dtree \text{ is an antichain}
        \Big\},\\
    \|(a_t)_{t\in \dtree}\|_{B^E}
    &=& \sup\Big\{ \| P_{\mathcal{A}} (a_t)_{t\in \dtree}\|_E \colon \mathcal{A}\subset \dtree \text{ is a branch}\Big\},
  \end{array}
\end{equation*}
respectively.  Of course, $B^{\ell^1}$ corresponds to the space $B$ defined above.  Moreover, we put
$D^E = (S^E)^*$. \smallskip

Our motivation for introducing and studying the stopping-time spaces is twofold.  First of all we
observe that their canonical unit vector bases constitute natural examples of the so-called
strategically reproducible bases introduced by Motakis, M\"uller, Schlumprecht, and the second-named
author in \cite{MR4145794} (all unexplained terminology will be discussed in subsequent sections),
where it was proved that the Haar basis of $L^1$ is strategically re\-pro\-du\-ci\-ble.  Strategic
reproducibility is intimately related to the factorisation property of a~Banach space with a
Schauder basis, which in turn often implies primarity of the space or uniqueness of the maximal
ideal of the algebra of operators on such a space.\smallskip

For a space $E$ with a $1$-subsymmetric Schauder basis indexed by $\dtree$ we denote by
\begin{itemize}
\item $(e_t)_{t\in\dtree}$ the standard Schauder basis in $S^E$ or $B^E$,
\item $(f_t)_{t\in\dtree}$ the associated biorthogonal functionals in $X^*$ for $X = S^E$ or
  $X = B^E$.
\end{itemize}
We are now ready to present the first main result.
\begin{mainth}\label{thm:factor}
  Let $E$ be a space with a $1$-subsymmetric Schauder basis and let $X$ denote either of the spaces
  $S^E$ or $B^E$.  Let $1\leqslant p, p'\leqslant\infty$ with $1/p + 1/p' = 1$.  If
  \begin{itemize}
  \item $(e_{k,t}\colon k\in\mathbb{N},\ t\in\dtree)$ denotes the standard Schauder basis of
    $\ell^p(X)$, and
  \item $(f_{k,t}\colon k\in\mathbb{N},\ t\in\dtree)$ denotes the corresponding biorthogonal
    functionals in $\ell^{p'}(X^*)$,
  \end{itemize}
  then the following assertions hold true:
  \begin{romanenumerate}
  \item\label{enu:thm:factor:i} $(e_t)_{t\in\dtree}$ is strategically reproducible.
  \item\label{enu:thm:factor:ii} $((e_t,f_t))_{t\in\dtree}$ is strategically supporting and has the
    factorisation property in $X\times X^*$.
  \item\label{enu:thm:factor:iii} $((e_{k,t},f_{k,t})\colon k\in\mathbb{N},\ t\in\dtree)$ is
    strategically supporting and has the factorisation property in $\ell^p(X)\times\ell^{p'}(X^*)$.
  \item\label{enu:thm:factor:iv} If the $1$-subsymmetric Schauder basis for $E$ is incomparably
    non-$c_0$ on antichains, then the system $((f_t,e_t))_{t\in\dtree}$ is strategically supporting
    and has the factorisation property in $D^E\times S^E$.
  \end{romanenumerate}
\end{mainth}

\begin{proof}
  Assertion~\eqref{enu:thm:factor:i} follows from \Cref{thm:strat-rep} and~\eqref{enu:thm:factor:ii}
  from \Cref{cor:strat-rep-1}.  \Cref{cor:strat-rep-2} proves~\eqref{enu:thm:factor:iii} and
  \Cref{cor:factor-D} shows~\eqref{enu:thm:factor:iv}.
\end{proof}

Apatsidis~\cite{apatsidis:2015} studied (bounded, linear) operators from a space $X$ with an
unconditional Schauder basis into $S^1$ and proved, among other things, that:
\begin{itemize}
\item the space $S^1$ is \emph{complementably homogeneous} in the sense that every isomorphic copy
  of $S^1$ in $S^1$ contains a further copy of $S^1$ that is, moreover complemented in $S^1$;
\item more generally, if an operator $T\colon X\to S^1$ fixes a copy of $S^1$, then $I_{S^1}$, the
  identity operator on $S^1$, factors through $T$, which means that $I_{S_1} = ATB$ for some
  operators $A\colon S^1\to X$ and $B\colon X\to S^1$;
\item if $S^1$ is isomorphic to the $\ell^1$-sum of a sequence of Banach spaces, then at least one
  summand therein is isomorphic to $S^1$---in particular, $S^1$ is primary.
\end{itemize}
The above results provide yet another example of resemblance between $S^1$ and $L^1$ as Enflo and
Starbird~\cite{enflo:starbird:1979} proved that the identity operator on $L^1$ factors through every
operator $T\colon L^1\to L^1$ that fixes a copy of $L^1$.\smallskip

Johnson and Dosev~\cite{dosev:johnson:2010} (see also \cite{chen2011commutators,dosev2013commutators}) studied the problem of which operators on a Banach space $X$ \emph{cannot} be commutators, that is, operators that are not expressible as $AB-BA$ for certain $A,B\in \mathscr{B}(X)$, where $\mathscr{B}(X)$ stands for the algebra of all operators on $X$. A direct application of Wintner's theorem (that the identity in a normed algebra is not a commutator) yields that operators of the form $\lambda I_X + K$ cannot be commutators for non-zero $\lambda$ and a compact operator $K\in \mathscr{B}(X)$ (in the general case, it is still the only known obstruction). For this, they considered the following
set of operators:
\begin{equation}\label{eq:Mx}
  \mathscr{M}_X
  = \big\{T\in \mathscr{B}(X)\colon I_X \neq ATB\;\; \big(A,B\in \mathscr{B}(X)\big)\big\},
\end{equation}
for which they observed that operators of the form $\lambda I + K$, where $\lambda$ is non-zero and
$K\in \mathscr{M}_X$ are not commutators as long as $\mathscr{M}_X$ is an ideal.  Clearly,
$\mathscr{M}_X$ is closed under multiplication from left and right, and it is (the unique maximal)
ideal of $\mathscr{B}(X)$ if and only if it is closed under addition. The case where $\mathscr{M}_X$
is indeed the unique maximal ideal of $\mathscr{B}(X)$ is of particular interest from the point of
view the theory of operator ideals; a list of spaces (containing many classical ones) for which
$\mathscr{M}_X$ is closed under addition may be found in \cite{kania:laustsen:2012}.  Usually the
proof of the fact that for a given space $X$ the set $\mathscr{M}_X$ is (or is not) closed under
addition are specific to idiosyncratic properties of the space.  \smallskip

As for $1\leqslant p < q <\infty$, the Banach spaces $\ell_p$ and $\ell_q$ are totally incomparable (every operator between them is strictly singular), the space $X = \ell_p \oplus \ell_q$ is a natural example of a space for which $\mathscr{M}_X$ is not an ideal (for details, see \cite{volkmann1976operatorenalgebren}). Algebras of operators on the Tsirelson or Schreier spaces of any finite order have infinitely many maximal ideals \cite{beanland2020closed} so in these cases the set $\mathscr{M}_X$ is not an ideal either.\smallskip

It follows from the
above-mentioned results by Apatsidis that $\mathscr{M}_{S^1}$ coincides with the set
$\mathscr{S}_{S^1}(S^1)$ comprising all $S^1$-singular operators, that is operators that do not fix
any copies of $S^1$.  As $S^1$ is complementably homogeneous $\mathscr{S}_{S^1}(S^1)$ (hence
$\mathscr{M}_{S^1}$) is the unique maximal ideal of $\mathscr{B}(S^1)$ (see~\cite[Corollary
2.3]{horvath:kania:2020}).\smallskip

In the present paper we propose a~rather general criterion for a space with a strategically reproducible Schauder basis that allows for concluding that $\mathscr{M}_X$ is indeed closed under
addition.  In order to state it, we require to introduce auxiliary definitions.\smallskip

Let $(X,Y, \langle \cdot, \cdot \rangle)$ be a dual pair of Banach spaces, that is, a pair of Banach spaces $X$ and $Y$, with a fixed non-degenerate bilinear form $\langle \cdot, \cdot \rangle$ from $X\times Y$ to the scalar field.  We say that a system
$((e_\gamma, f_\gamma))_{\gamma\in \Gamma}$ in $X\times Y$:
\begin{itemize}
\item is \emph{biorthogonal}, whenever
  $\langle e_{\gamma_1}, f_{\gamma_2}\rangle = \delta_{\gamma_1, \gamma_2}$
  $(\gamma_1, \gamma_2\in\Gamma)$;
\item has the \emph{positive factorisation property (in $X\times Y$)}, whenever for every
  $T\in \mathscr{B}(X)$ with $\inf_{\gamma\in \Gamma}\langle Te_\gamma, f_\gamma \rangle > 0$ one
  has $T\notin \mathscr{M}_X$.
\end{itemize}

\begin{rem}\label{rem:factorisation-property}
  In~\cite{MR4145794}, the notion of the factorisation property for
  $((e_\gamma, f_\gamma))_{\gamma\in \Gamma}$ was introduced in the following way: for every
  operator $T\in \mathscr{B}(X)$ with
  $\inf_{\gamma\in \Gamma}|\langle Te_\gamma, f_\gamma \rangle| > 0$ one has
  $T\notin \mathscr{M}_X$.  Clearly, the factorisation property implies the positive factorisation
  property.  For unconditional Schauder bases $(e_\gamma)_{\gamma\in\Gamma}$ and the associated
  coordinate functionals $(f_\gamma)_{\gamma\in\Gamma}$, the factorisation property is equivalent to
  the positive factorisation property.
\end{rem}

We are now ready to state the general criterion for $\mathscr{M}_X$ being closed under addition
(hence being the unique maximal of $\mathscr{B}(X)$).

\begin{mainth}\label{thm:max-ideal:2}
  Let $(X,Y, \langle \cdot, \cdot \rangle)$ be a dual pair of Banach spaces.  Suppose that
  \begin{itemize}
  \item $(e_n)_{n=1}^\infty$ is a basis of $X$ with respect to the topology $\sigma(X,Y)$,
  \item $(f_n)_{n=1}^\infty$ is a basis of $Y$ with respect to the topology $\sigma(Y,X)$,
  \item the system $((e_n, f_n))_{n=1}^\infty$ is biorthogonal,
  \item there exists $c > 0$ such that
    \begin{equation}\label{eq:20}
      c \|x\|
      \leqslant \sup_{\|y\|\leqslant 1} \langle x, y \rangle
      \leqslant \|x\|
      \qquad (x\in X),
    \end{equation}
  \item $((e_n, f_n))_{n=1}^\infty$ is strategically supporting,
  \item $((e_n, f_n))_{n=1}^\infty$ has the positive factorisation property.
  \end{itemize}
  Then $\mathscr{M}_X$ is the unique closed proper maximal ideal of $\mathscr{B}(X)$.
\end{mainth}
A proof for \Cref{thm:max-ideal:2} is provided in \Cref{sec:uniq-maxim-ideals}.

The framework of \Cref{thm:max-ideal:2} is general enough to encompass dual spaces that have weak*
Schauder bases such as the space $\ell^\infty$.  With the aid of this result, we shall prove that
the stopping-time spaces, their duals, and various other related spaces $X$ have the property that
$\mathscr{M}_X$ is closed under addition (and hence the unique maximal ideal of $\mathscr{B}(X)$).

\begin{mainth}\label{th:c}
  Let $X$ be any Banach space and let $E$ denote a space with a normalised $1$-subsymmetric Schauder
  basis.  Then $\mathscr{M}_X$ is the unique maximal ideal of $\mathscr{B}(X)$ in the following
  cases:
  \begin{itemize}
  \item $X = S^E$, $X = B^E$;
  \item $X = \ell^p(S^E)$ or $X=\ell^p(B^E)$, $1\leqslant p\leqslant\infty$;
  \item $X = D^E = (S^E)^*$, whenever the $1$-subsymmetric Schauder basis of $E$ is incomparably
    non-$c_0$ on antichains.
  \end{itemize}
  Consequently, the spaces $\ell^p(S^E)$, $\ell^p(B^E)$ $(1\leqslant p\leqslant\infty)$ are primary.
\end{mainth}

\begin{proof}
  The first two assertions follow immediately from combining the results \Cref{cor:strat-rep-1},
  \Cref{cor:factor-D}, \Cref{cor:strat-rep-2}, and \Cref{thm:max-ideal:2}.\smallskip

  To see the last claim holds true, first note that for any Banach space $X$, we have that
  $\ell^p ( \ell^p(X) )$ is isomorphic to $\ell^p(X)$.  Now, let $X = \ell^p(S^E)$ or
  $X=\ell^p(B^E)$.  Since $\mathscr{M}_X$ is closed under addition, for a given operator
  $T\in\mathscr{B}(X)$ it is impossible that both $T$ and $I-T$ are in $\mathscr{M}_X$ (otherwise we
  would have $I\in\mathscr{M}_X$).  Thus, whenever $P$ is a projection on $X$, the identity $I_X$ on
  $X$ either factors through $P$ or $I_X-P$.  This implies that either the range of $X$ or the range
  of $I_X-P$ contains a complemented copy of $X$.  Finally, by using Pe{\l}czy{\'n}ski's
  decomposition method (\cite{pelczynski:1960}; see also~\cite[II.B.24]{wojtaszczyk:1991}) the proof
  is complete.
\end{proof}
  
The above-described techniques find their applications beyond the class of stopping-time spaces.
More specifically, \Cref{cor:max-ideals-Hp-bmo-SLinfty}, \Cref{cor:max-ideals-lplq}, and
\Cref{cor:max-ideals-HpHq} expand the list of spaces $X$ collected in \cite{kania:laustsen:2012} for
which $\mathscr{M}_X$ is the unique maximal ideal of $\mathscr{B}(X)$; we record these results
jointly below.

\begin{mainth}
  Let $X$ be one of the spaces:
  \begin{itemize}
  \item $L^p$ $(1\leqslant p < \infty)$, $H^1$, $\bmo$, $\mathrm{SL^\infty}$
  \item $\ell^p(\ell^q)$ for $1\leqslant p\leqslant\infty$ and $1 < q < \infty$, or
  \item $H^1(H^1)$ or $L^p(L^q)$ for $1 < p,q < \infty$.
  \end{itemize}
  Then $\mathscr{M}_X$ is the unique maximal ideal of $\mathscr{B}(X)$.
\end{mainth}

\section{Preliminaries}
We use standard Banach space terminology that is mostly in-line with
\cite{lindenstrauss-tzafriri:1977}.  We consider real Banach spaces, however the results, with some
effort, may be extended to complex scalars too.  By an \emph{operator} we understand a bounded
linear map acting between normed spaces.  A Banach space $X$ is \emph{primary} as long as whenever
$X = W \oplus V$, then at least one of the subspaces $W$, $V$ is isomorphic to $X$.

Let $X$ be a Banach space and $1\leqslant p\leqslant \infty$.  We define
\begin{equation*}
  \ell^p(X)
  = \bigl\{ (x_n)_{n=1}^\infty\colon x_n\in X,\ \|(x_n)_{n=1}^\infty\|_{\ell^p(X)} < \infty \bigr\},
\end{equation*}
where the norm $\|\cdot\|_{\ell^p(X)}$ is given by
\begin{equation*}
  \|(x_n)_{n=1}^\infty\|_{\ell^p(X)}
  = \bigl\| \bigl(\|x_n\|_X\bigr)_{n=1}^\infty\bigr\|_{\ell^p}.
\end{equation*}

Let $(X,Y,\langle\cdot,\cdot\rangle)$ be a dual pair of Banach spaces and let $\sigma(X,Y)$ denote
the weakest topology such that all the maps $\langle \cdot, y\rangle$ ($y\in Y$) are continuous.  A
sequence $(e_n)_{n=1}^\infty$ in $X$ is a \emph{basis of $X$ with respect to the topology
  $\sigma(X,Y)$}, whenever for every $x\in X$ there exists a~unique sequence of scalars
$(a_n)_{n=1}^\infty$ such that
\begin{equation}\label{eq:basis}
  x = \sum_{n=1}^\infty a_n e_n,
\end{equation}
where the above series converges in the $\sigma(X,Y)$ topology.
\begin{itemize}
\item If $Y = X^*$, that is when $\sigma(X,X^*)$ is the usual weak topology, we then call bases
  \emph{Schauder bases} (by the Orlicz--Pettis theorem, the series \eqref{eq:basis} converges in the
  norm topology).
\item When the weak* topology $\sigma(X^*, X)$ is considered, we call the corresponing bases
  \emph{weak* Schauder bases}.
\item If $\|e_n\|_X = 1$ ($n\in\mathbb{N}$), the basis $(e_n)_{n=1}^\infty$ is called
  \emph{normalised}.
\item Let $(x_n)_{n=1}^\infty$ and $(\xi_n)_{n=1}^\infty$ be sequences in $X$.  We say that
  $(x_n)_{n=1}^\infty$ is $C$-\emph{dominated} by $(\xi_n)_{n=1}^\infty$, whenever for all scalar
  sequences $(a_n)_{n=1}^\infty$
  \begin{equation*}
    \sum_{n=1}^\infty a_n x_n
    \text{ converges whenever }
    \sum_{n=1}^\infty a_n \xi_n
    \text{ converges}
  \end{equation*}
  (both series convergence is with respect to the $\sigma(X,Y)$ topology) and for all sequences
  $(a_n)_{n=1}^\infty$ such that $\sum_{n=1}^\infty a_n \xi_n$ converges in the $\sigma(X,Y)$
  topology we have
  \begin{equation*}
    \Bigl\|\sum_{n=1}^\infty a_n x_n\Bigr\|
    \leqslant C \Bigl\|\sum_{n=1}^\infty a_n \xi_n\Bigr\|.
  \end{equation*}
  If $(x_n)_{n=1}^\infty$ is $C$-dominated by $(\xi_n)_{n=1}^\infty$ for some $C > 0$, then we say
  that \emph{$(x_n)_{n=1}^\infty$ is dominated by $(\xi_n)_{n=1}^\infty$}.
\item We say that \emph{$(x_n)_{n=1}^\infty$ is $C$-equivalent to $(\xi_n)_{n=1}^\infty$}, whenever there are
  non-negative constants $C_1, C_2$ with $C = C_1\cdot C_2$ such that $(x_n)_{n=1}^\infty$ is
  $C_1$-dominated by $(\xi_n)_{n=1}^\infty$ and $(\xi_n)_{n=1}^\infty$ is $C_2$-dominated by
  $(x_n)_{n=1}^\infty$.  If $C_1 = C_2 = \sqrt{C}$, we say that \emph{$(x_n)_{n=1}^\infty$ is impartially
    $C$-equivalent to $(\xi_n)_{n=1}^\infty$}.  When $(x_n)_{n=1}^\infty$ is $C$-equivalent to
  $(\xi_n)_{n=1}^\infty$ for some $C > 0$, we simply say that \emph{$(x_n)_{n=1}^\infty$ is equivalent to
    $(\xi_n)_{n=1}^\infty$}.
\end{itemize}

\subsection{The dyadic tree}
\label{sec:dyadic-tree}

Let $\dtree$ denote the rooted dyadic tree, that is, the tree comprising all finite sequences of
$0$s and $1$s, with the root denoted by $\varnothing$.  Thus,
$\dtree = \bigcup_{n=0}^\infty \{0,1\}^n$ with the convention $\{0,1\}^0 = \{\varnothing\}$.
Moreover, we write $2^{\leqslant n} = \bigcup_{j=1}^n \{0,1\}^n$.  For $s\in 2^{\leqslant n}$ and
$m\leqslant n$, we write $s|_m = (s(1), \ldots, s(m))$ when $m>0$ and $s|_0 = \varnothing$.  As
every $s\in \dtree$ belongs to a uniquely determined set $2^{\leqslant n}$, in such a case we call
$n$ the \emph{length} of $s$ and denote it by $|s|$.  For $t,s\in \dtree$ we denote by
$t^\smallfrown s$ the \emph{concatenation} of $t$ and $s$:
$t^\smallfrown s = (t(1), \ldots, t(|t|), s(1), \ldots, s(|s|))$.  The set $\dtree$ is naturally
ordered by the initial-segment partial ordering: $s\sqsubseteq t$ whenever $|s|\leqslant |t|$ and
$s = t|_{|s|}$.  We say that $s,t\in\dtree$ are \emph{($\sqsubseteq$-)incomparable}, if neither
$s\sqsubseteq t$ nor $t\sqsubseteq s$.  For $A,B\subseteq\dtree$, we write $A\sqsubseteq B$ whenever
$s\sqsubseteq t$ for all $s\in A$ and $t\in B$.  We say that $A,B\subseteq\dtree$ are
\emph{($\sqsubseteq$-)incomparable} if $s$ and $t$ are $\sqsubseteq$-incomparable for all $s\in A$
and $t\in B$. \smallskip

A~\emph{branch} of $\dtree$ is a maximal linearly ordered (with respect to `$\sqsubseteq$') subset
of $\dtree$ .  Let us denote by $\beta$ the set of all branches in $\dtree$.  An \emph{antichain} in
$\dtree$ is any subset comprising pairwise $\sqsubseteq$-incomparable elements.  Given a node
$t\in \dtree\setminus\{\varnothing\}$, let $\tilde{t}$ denote the unique node such that
$t = \tilde{t}^\smallfrown \alpha$ for some $\alpha\in\{0,1\}$.  We say that
$\mathscr{T}\subset\dtree$ is a \emph{subtree}, if $\mathscr{T}$ is order isomorphic to $\dtree$
with respect to the partial order `$\sqsubseteq$'.

\smallskip

Let $s,t\in \dtree \setminus\{\varnothing\}$ denote two $\sqsubseteq$-incomparable nodes and suppose
that $n$ is the unique maximal integer such that $0\leq n\leq |s|,|t|$ and $s(n) = t(n)$ (with the
agreement that $r(0) = \varnothing$ for all $r\in \dtree $).  Observe that $|s|,|t| > n$ (otherwise
we would have $s\sqsubseteq t$ or $t\sqsubseteq s$) and that $s(n+1)\neq t(n+1)$ by the maximality
of $n$.  We say that \emph{$s$ is to the left of $t$} if $s(n+1) = 0$, and if $s(n+1) = 1$, we say
that \emph{$s$ is to the right of $t$}.  Given two $\sqsubseteq$-incomparable sets
$A,B\subseteq \dtree \setminus\{\varnothing\}$, we say that \emph{$A$ is to the left (right) of
  $B$}, whenever all elements in $A$ are to the left (right) of all elements in $B$.  By `$<$', we
denote the \emph{standard linear order on $\dtree$}, \emph{i.e.}, $s < t$ if and only if $|s| < |t|$
or if $|s| = |t|$ and $s$ is to the left of $t$.  For $A,B\subset\dtree$, we write $A < B$ whenever
$s < t$ for all $s\in A$ and $t\in B$.  Let $\mathcal{O}\colon \dtree \to\mathbb{N}$ denote the
bijective order preserving map with respect to the standard linear order `$<$' on the tree $\dtree $
and the natural order on $\mathbb{N}$.  If $t\in \dtree $ and $k\in\mathbb{Z}$, we write $t + k$ for
the unique $t'\in \dtree $ such that $\mathcal{O}(t') = \mathcal{O}(t) + k$, if it exists.  We say
that a subtree $\mathscr{T}$ of $\dtree$ is \emph{linearly order isomorphic} to $\dtree$, if
$\mathscr{T}$ is order isomorphic with respect to the standard linear order.\smallskip

Let $E$ be a space with a normalised Schauder basis indexed by $\dtree $, say $(e_t)_{t\in \dtree}$.
We say that $(e_t)_{t\in \dtree}$ is $1$\emph{-unconditional}, whenever for all finitely supported
sequences of scalars $(a_t)_{t\in \dtree}$ and $(\gamma_t)_{t\in \dtree}$ one has
\begin{equation*}\label{eq:unconditional}
  \Big\| \sum_{t\in \dtree} \gamma_t a_t e_t \Big\|_E
  \leqslant \sup_{t\in \dtree} |\gamma_t| \Big\| \sum_{t\in \dtree} a_t e_t \Big\|_E.
\end{equation*}
Moreover, we say that $(e_t)_{t\in \dtree}$ is \emph{$1$-spreading}, if $(e_t)_{t\in \dtree}$ is
$1$-equivalent to each of its increasing subsequences (with respect to the standard linear order),
\emph{i.e.}, if for all finitely supported sequences of scalars $(a_t)_{t\in \dtree}$
\begin{equation*}\label{eq:spreading}
  \Bigl\|\sum_{t\in \dtree} a_t e_{t}\Bigr\|_E
  = \Bigl\|\sum_{t\in \dtree} a_t e_{s_t}\Bigr\|_E,
\end{equation*}
whenever $(s_t)_{t\in \dtree}$ is such that $s_{t_1} < s_{t_2}$ if $t_1 < t_2$.  We say that the
sequence $(e_t)_{t\in \dtree}$ is \emph{$1$-subsymmetric} if it is $1$-unconditional as well as
$1$-spreading.

\section{The factorisation property and the uniqueness of maximal ideals}
\label{sec:uniq-maxim-ideals}

Let $E$ be a Banach space with a normalised Schauder basis $(e_n)_{n=1}^\infty$.  It is tempting to
speculate that for an operator $T\in \mathscr{B}(E)$, the condition
$\inf_n |\langle Te_n, e_n^*\rangle| > 0$ is a~sufficient condition for factoring the identity,
\emph{i.e.}, for generating the whole $\mathscr{B}(E)$ as an ideal by $T$.  Using Gowers' space with
an unconditional Schauder basis, one can construct a counterexample
(\cite[Theorem~2.1]{laustsen:lechner:mueller:2015}).  However, for many classical spaces such as
$\bmo$, $\mathrm{SL^\infty}$, $\ell^p(\ell^q)$, and other, the condition
$\inf_n |\langle Te_n, e_n^*\rangle| > 0$ is indeed sufficient (see \Cref{sec:applications} for
details).

\begin{dfn}\label{dfn:strat-annihil}
  Let $(X,Y,\langle\cdot,\cdot\rangle)$ be a dual pair of Banach spaces.  Suppose that
  $(e_n)_{n=1}^\infty$ and $(f_n)_{n=1}^\infty$ are sequences in $X$ and $Y$, respectively.  We say
  that \emph{$((e_n,f_n))_{n=1}^\infty$ is almost annihilating for a set
    $\mathcal{A}\subset \mathscr{B}(X)$}, whenever for every $T\in \mathcal{A}$ and $\eta > 0$ there
  exist sequences $(x_n)_{n=1}^\infty$ in $X$ and $(y_n)_{n=1}^\infty$ in $Y$, respectively, such
  that
  \begin{enumerate}[(i)]
  \item\label{enu:dfn:strat-annihil:i} $(x_n)_{n=1}^\infty$ is dominated by $(e_n)_{n=1}^\infty$;
  \item\label{enu:dfn:strat-annihil:ii} $(y_n)_{n=1}^\infty$ is dominated by $(f_n)_{n=1}^\infty$;
  \item\label{enu:dfn:strat-annihil:iii}
    $\inf_{n\in \mathbb N} \langle x_n, y_n\rangle \geqslant 1$;
  \item\label{enu:dfn:strat-annihil:iv}
    $\sup_{n\in \mathbb N} \langle T x_n, y_n\rangle \leqslant\eta$.
  \end{enumerate}
\end{dfn}

The notion of almost annihilation is a property that in tandem with the factorisation property
yields that $\mathscr{M}_X$ is an ideal.

\begin{thm}\label{thm:max-ideal:1}
  Let $(X,Y,\langle\cdot,\cdot\rangle)$ be a dual pair of Banach spaces.  Let $(e_n)_{n=1}^\infty$
  be a basis for $X$ with respect to the $\sigma(X,Y)$ topology, let $(f_n)_{n=1}^\infty$ be a basis
  for space $Y$ with respect to the $\sigma(Y,X)$ topology and assume there exists a constant
  $c > 0$ such that
  \begin{equation}\label{eq:4}
    c \|x\|
    \leqslant \sup_{\|y\|\leqslant 1} \langle x, y \rangle
    \leqslant \|x\|
    \qquad (x\in X).
  \end{equation}
  If $((e_n,f_n))_{n=1}^\infty$ is almost annihilating $\mathscr{M}_X$ and has the positive
  factorisation property, then $\mathscr{M}_X$ is the unique closed proper maximal ideal of
  $\mathscr{B}(X)$.
\end{thm}

\begin{proof}
  Let $0 < \eta < 1$, $S\in\mathscr{M}_X$, $T\in \mathscr{B}(X)$ and $S+T\notin\mathscr{M}_X$.  Then
  there exist operators $A,B\in \mathscr{B}(X)$ such that $I_X = B(S+T)A$.  Since $BSA$ must be in
  $\mathscr{M}_X$ and $((e_n,f_n))_{n=1}^\infty$ is almost annihilating for $\mathscr{M}_X$, we can
  find sequences $(x_n)_{n=1}^\infty$ in $X$ dominated by $(e_n)_{n=1}^\infty$ and
  $(y_n)_{n=1}^\infty$ in $Y$ dominated by $(f_n)_{n=1}^\infty$ such that
  \begin{equation*}
    \langle x_n, y_n\rangle \geqslant 1
    \quad\text{and}\quad
    \langle BSA x_n, y_n\rangle\leqslant \eta
    \qquad (n\in\mathbb{N}).
  \end{equation*}
  Hence,
  \begin{equation}\label{eq:17}
    1
    \leqslant \langle x_n, y_n\rangle
    = \langle B(S+T)Ax_n, y_n\rangle
    \leqslant \eta + \langle BTAx_n y_n\rangle
    \qquad (n\in\mathbb{N}).
  \end{equation}
  Define the operators $L\colon X\to X$ and $R\colon Y\to Y$ as the linear extensions of
  \begin{equation*}
    Le_n = x_n
    \quad\text{and}\quad
    Rf_n = y_n
    \qquad (n\in\mathbb{N}).
  \end{equation*}
  By~\eqref{enu:dfn:strat-annihil:i} and~\eqref{enu:dfn:strat-annihil:ii} in
  \Cref{dfn:strat-annihil}, the operators $L$ and $R$ are well defined and bounded.  Next, put
  $U = R^*BTAL$, where $R^*$ denotes the unique operator such that
  $\langle R^* x, y\rangle = \langle x, Ry\rangle$ for all $x\in X$, $y\in Y$.  Note that
  \eqref{eq:4} yields $R^*\in \mathscr{B}(X)$ and by~\eqref{eq:17} we obtain
  \begin{equation*}
    \inf_n\langle Ue_n,f_n\rangle
    = \inf_n\langle BTAx_n,y_n\rangle
    \geqslant 1 - \eta > 0.
  \end{equation*}
  Thus, since $((e_n,f_n))_{n=1}^\infty$ has the positive factorisation property, we obtain
  $U\notin \mathscr{M}_X$ and consequently $T\notin\mathscr{M}_X$.  We showed that $\mathscr{M}_X$
  is closed under addition, and so it is the unique maximal ideal of $\mathscr{B}(X)$;
  see~\cite[Section~5]{dosev:johnson:2010} for details.
\end{proof}

\begin{dfn}\label{dfn:strat-supp}
  Let $(X,Y,\langle\cdot,\cdot\rangle)$ be a dual pair of Banach spaces.  Let $(e_n)_{n=1}^\infty$
  be a basis for $X$ with respect to the $\sigma(X,Y)$ topology and let $(f_n)_{n=1}^\infty$ be a
  basis for $Y$ with respect to the $\sigma(Y,X)$ topology.  We say that
  \emph{$((e_n,f_n))_{n=1}^\infty$ is strategically supporting (in $X\times Y$)} if for all
  $\eta > 0$ and all partitions $N_1,N_2$ of $\mathbb{N}$ there exists $i\in\{1,2\}$ and
  \begin{align*}
    \exists&\ \text{finite}\ E_1\subset N_i\;\exists(\lambda_j^1)_j, (\mu_j^1)_j\in\mathbb{R}^{E_1}\;
             \forall (\varepsilon_j^1)_j\in\{\pm 1\}^{E_1}\\
    \exists&\ \text{finite}\   E_2\subset N_i\;\exists(\lambda_j^2)_j, (\mu_j^2)_j\in\mathbb{R}^{E_2}\;
             \forall (\varepsilon_j^2)_j\in\{\pm 1\}^{E_2}\\
           &\vdots\\
    \exists&\ \text{finite}\  E_k\subset N_i\;\exists(\lambda_j^k)_j, (\mu_j^k)_j\in\mathbb{R}^{E_k}\;
             \forall (\varepsilon_j^k)_j\in\{\pm 1\}^{E_k}\\
           &\vdots
  \end{align*}
  such that if we define
  \begin{equation*}
    x_k
    = \sum_{j\in E_k}\varepsilon_j^k\lambda_j^ke_j
    \quad\text{and}\quad
    y_k
    = \sum_{j\in E_k}\varepsilon_j^k\mu_j^kf_j
    \qquad (k\in\mathbb{N}),
  \end{equation*}
  we have that
  \begin{enumerate}[(i)]
  \item\label{enu:dfn:strat-supp:i} $(x_k)_{k=1}^\infty$ is dominated by $(e_k)_{k=1}^\infty$;
  \item\label{enu:dfn:strat-supp:ii} $(y_k)_{k=1}^\infty$ is dominated by $(f_k)_{k=1}^\infty$;
  \item\label{enu:dfn:strat-supp:iii} $1 \leqslant \langle x_k, y_k\rangle \leqslant 1 + \eta$,
    $k\in\mathbb{N}$;
  \item\label{enu:dfn:strat-supp:iv} $\lambda_j^k\mu_j^k\geqslant 0$, $k\in\mathbb{N}$, $j\in E^k$.
  \end{enumerate}
\end{dfn}

We are now prepared to prove \Cref{thm:max-ideal:2}.
\begin{proof}[Proof of \Cref{thm:max-ideal:2}]
  Let $\eta > 0$ and $T\in\mathscr{M}_X$ and define
  $N_1 = \{n\in\mathbb{N}\colon \langle T e_n, f_n\rangle \leqslant \eta/(1+\eta)\}$ and
  $N_2 = \mathbb{N}\setminus N_1$.  Since $((e_n,f_n))_{n=1}^\infty$ is strategically supporting,
  there exist sequences $(x_k)_{k=1}^\infty$ in $X$ and $(y_k)_{k=1}^\infty$ in $Y$ given by
  \begin{equation*}
    x_k
    = \sum_{j\in E_k}\varepsilon_j^k\lambda_j^ke_j
    \quad\text{and}\quad
    y_k
    = \sum_{j\in E_k}\varepsilon_j^k\mu_j^kf_j
    \qquad (k\in\mathbb{N})
  \end{equation*}
  that satisfy~\eqref{enu:dfn:strat-supp:i}--\eqref{enu:dfn:strat-supp:iv} in \Cref{dfn:strat-supp},
  and there exists $i\in\{1,2\}$ such that $E_k\subset N_i$ for all $k\in\mathbb{N}$.  Additionally,
  we note that in the $k^{\text{th}}$ step of \Cref{dfn:strat-supp}, we are free to choose the signs
  $(\varepsilon_j^k)_j\in\{\pm 1\}^{E_k}$.  Let $\cond_\varepsilon$ denote the average over all
  those possible choices of signs and observe that
  \begin{equation}\label{eq:18}
    \begin{aligned}
      \cond_\varepsilon \langle T x_k, y_k\rangle &= \cond_\varepsilon \Bigl\langle T \sum_{j\in
        E_k} \varepsilon_j\lambda_j^k e_j, \sum_{l\in E_k} \varepsilon_l\mu_l^k f_l\Bigr\rangle =
      \cond_\varepsilon \sum_{j,l\in E_k} \varepsilon_j\varepsilon_l\lambda_j^k \mu_l^k
      \langle T e_j, f_l\rangle\\
      &= \sum_{j\in E_k} \lambda_j^k \mu_j^k \langle T e_j, f_j\rangle.
    \end{aligned}
  \end{equation}
  Note that we only need to verify \Cref{dfn:strat-annihil}~\eqref{enu:dfn:strat-annihil:iv}.

  Now suppose to the contrary that $i=2$.  In this case, using~\eqref{eq:18} we would have picked
  signs $(\varepsilon_j^k)_j\in\{\pm 1\}^{E_k}$ in the $k^{\text{th}}$ step of \Cref{dfn:strat-supp}
  so that
  \begin{equation*}
    \langle T x_k, y_k\rangle
    \geqslant \sum_{j\in E_k} \lambda_j^k \mu_j^k \langle T e_j, f_j\rangle
    \qquad (k\in\mathbb{N}).
  \end{equation*}
  Since $E_k\subset N_2$, $\lambda_j^k \mu_j^k\geqslant 0$, $j\in E_k$, by the biorthogonality of
  $((e_n,f_n))_{n=1}^\infty$ and $\langle x_k, y_k\rangle\geqslant 1$, it follows that
  \begin{equation}\label{eq:19}
    \langle T x_k, y_k\rangle
    \geqslant \eta/(1+\eta) \sum_{j\in E_k} \lambda_j^k \mu_j^k
    = \eta/(1+\eta) \langle x_k, y_k\rangle
    \geqslant \eta/(1+\eta)
    \qquad (k\in\mathbb{N}).
  \end{equation}
  We define $A\colon X\to X$ and $B\colon Y\to Y$ as the linear extensions of
  \begin{equation*}
    A e_k = x_k
    \quad\text{and}\quad
    B f_k = y_k
    \qquad (k\in\mathbb{N})
  \end{equation*}
  and note that since $(x_k)_{k=1}^\infty$ is dominated by $(e_k)_{k=1}^\infty$ and
  $(y_k)_{k=1}^\infty$ is dominated by $(f_k)_{k=1}^\infty$, $A$ and $B$ are well defined and
  bounded.  Let $B^*$ denote the unique operator such that
  $\langle B^* x, y\rangle = \langle x, By\rangle$ for all $x\in X$, $y\in Y$, and note that $B^*$
  is bounded by~\eqref{eq:20}.  Put $S = B^*TA$ and observe that by~\eqref{eq:19}, we obtain
  $\langle S e_k, f_k\rangle = \langle T x_k, y_k\rangle\geqslant \eta/(1+\eta) > 0$ for all
  $k\in\mathbb{N}$.  Since $((e_n,f_n))_{n=1}^\infty$ also has the factorisation property, we have
  $S\notin\mathscr{M}_X$; hence, $T\notin\mathscr{M}_X$, which contradicts our assumption
  $T\in\mathscr{M}_X$.

  Thus $i=1$.  In this case, using~\eqref{eq:18} we would have picked signs
  $(\varepsilon_j^k)_j\in\{\pm 1\}^{E_k}$ in the $k^{\text{th}}$ step of \Cref{dfn:strat-supp} such
  that
  \begin{equation*}
    \langle T x_k, y_k\rangle
    \leqslant \sum_{j\in E_k} \lambda_j^k \mu_j^k \langle T e_j, f_j\rangle
    \qquad (k\in\mathbb{N}).
  \end{equation*}
  Since $E_k\subset N_1$, $\lambda_j^k \mu_j^k\geqslant 0$, $j\in E_k$, as
  $((e_n,f_n))_{n=1}^\infty$ is biorthogonal and $\langle x_k, y_k\rangle\leqslant 1+\eta$, it
  follows that
  \begin{equation*}
    \langle T x_k, y_k\rangle
    \leqslant \eta/(1+\eta) \sum_{j\in E_k} \lambda_j^k \mu_j^k
    = \eta/(1+\eta) \langle x_k, y_k\rangle
    \leqslant \eta
    \qquad (k\in\mathbb{N}).
    \qedhere
  \end{equation*}
\end{proof}

\section{Factorisation of the identity in \texorpdfstring{$S^E$}{S^E} and
  \texorpdfstring{$B^E$}{B^E}}
\label{sec:fact-stopp-time}
The present section specialises to Banach spaces of the form $S^E$ and $B^E$ to which we shall apply
the just-established results.  For this, let us fix a Banach space $E$ with a normalised
$1$-subsymmetric Schauder basis $(e_t)_{t\in \dtree}$; we denote by $(e_t^*)_{t\in \dtree}$ the
associated biorthogonal functionals.

\subsection{Linearly order-isomorphic subtrees}
\label{sec:order-isom-subtr}

We begin the present section with a result that is known to the experts in Ramsey theory, however
for the sake of completeness, we include its proof.

\begin{lem}\label{lem:subtrees}
  Let $\mathscr{S}$ be a subset of $\dtree$.  Then either $\mathscr{S}$ or
  $\dtree\setminus\mathscr{S}$ contains a subtree $\mathscr{T}$ which is linearly order-isomorphic
  to $\dtree$.
\end{lem}

\begin{proof}
  Given a subset $\mathscr{S}\subset\dtree$, either $\mathscr{S}$ or $\dtree\setminus\mathscr{S}$
  contains a subtree $\mathscr{T}'$~\cite[Proposition~4]{MR709696}
  (see~\cite[Proposition~3.a.9]{MR1474498} for a proof).  By inductively replacing nodes in
  $\mathscr{T}'$ which violate the standard linear order with conforming successors, we can find a
  subtree $\mathscr{T}\subset\mathscr{T}'$ as claimed.
\end{proof}

\begin{lem}\label{lem:basic-operators}
  Let $X$ denote either of the spaces $S^E$ or $B^E$.  Given a linearly order isomorphic subtree
  $\mathscr{T} = \{s_t\colon t\in\dtree\}$, we define pointwise the operators $B,Q\colon X\to X$ by
  \begin{equation}\label{eq:lem:basic-operators:0}
    Bx
    = \sum_{t\in\dtree} \langle x, e_t^*\rangle e_{s_t}
    \qquad\text{and}\qquad
    Qx
    = \sum_{t\in\dtree} \langle x, e_{s_t}^*\rangle e_t
    \qquad (x\in X).
  \end{equation}
  Then we have
  \begin{equation}\label{eq:lem:basic-operators:1}
    QB = I_{X}
    \qquad\text{and}\qquad
    \|B\| = \|Q\| = 1,
  \end{equation}
  where $I_{X}$ denotes the identity operator on $X$.
\end{lem}

\begin{proof}
  Firstly, we note that
  \begin{equation}\label{eq:1}
    QB e_t =  e_t
    \qquad (t\in \dtree )
  \end{equation}
  and hence $QB=I_{X}$.

  Secondly, for every finite sequence of scalars $(a_t)_{t\in\dtree}$, we have that
  \begin{equation}\label{eq:22}
    \Bigl\| B\sum_{t}a_t e_t\Bigr\|_{X}
    = \Bigl\| \sum_{t}a_t e_{s_t}\Bigr\|_{X}
    = \sup_{\mathcal{C}} \Bigl\|\sum_{t:s_t\in\mathcal{C}} a_t e_{s_t}\Bigr\|_E.
  \end{equation}
  If $X = S^E$, the supremum above extends over all antichains $\mathcal{C}\subset\dtree$, and if
  $X=B^E$, the supremum extends over all branches $\mathcal{C}\subset\dtree$.  Since our subtree is
  linearly order isomorphic to $\dtree$, \emph{i.e.}, $s_{t_1} < s_{t_2}$ whenever $t_1 < t_2$, we
  obtain by the $1$-subsymmetry of $(e_t)_{t\in\dtree}$ that
  \begin{equation*}
    \Bigl\|\sum_{t:s_t\in\mathcal{C}} a_t e_{s_t}\Bigr\|_E
    \leqslant \Bigl\|\sum_{t:s_t\in\mathcal{C}} a_t e_t\Bigr\|_E.
  \end{equation*}
  Combining the latter estimate with~\eqref{eq:22} yields
  \begin{equation*}
    \Bigl\| B\sum_{t}a_t e_t\Bigr\|_{X}
    \leqslant \sup_{\mathcal{C}} \Bigl\|\sum_{t:s_t\in\mathcal{C}} a_t e_t\Bigr\|_E.
  \end{equation*}
  Note that since the set $\{t\colon s_t\in\mathcal{C}\}$ is either an antichain or a branch
  (depending on whether $X=S^E$ or $X=B^E$), the latter estimate together with the definition of the
  norm in $X$ yield $\|B\|\leqslant 1$, as claimed.

  Thirdly, observe that
  \begin{equation}\label{eq:26}
    \Bigl\| Q\sum_{t}a_t e_t\Bigr\|_{X}
    = \Bigl\|
    \sum_{t} a_{s_t} e_t
    \Bigr\|_{X}
    = \sup_{\mathcal{C}} \Bigl\|
    \sum_{t\in\mathcal{C}} a_{s_t} e_t
    \Bigr\|_E,
  \end{equation}
  where in case $X = S^E$, the supremum above extends over all antichains
  $\mathcal{C}\subset\dtree$, and if $X=B^E$, the supremum extends over all branches
  $\mathcal{C}\subset\dtree$.  Since $\mathscr{T}$ is linearly order isomorphic to $\dtree$, we
  obtain by the $1$-subsymmetry of $(e_t)_{t\in\dtree}$ that
  \begin{equation*}
    \Bigl\|
    \sum_{t\in\mathcal{C}} a_{s_t} e_t
    \Bigr\|_E
    \leqslant \Bigl\|
    \sum_{t\in\mathcal{C}} a_{s_t} e_{s_t}
    \Bigr\|_E.
  \end{equation*}
  So far, together with~\eqref{eq:26}, we showed that
  \begin{equation*}
    \Bigl\| Q\sum_{t}a_t e_t\Bigr\|_{X}
    \leqslant \sup_{\mathcal{C}}\Bigl\|
    \sum_{t\in\mathcal{C}} a_{s_t} e_{s_t}
    \Bigr\|_E,
  \end{equation*}
  If $X=S^E$, then $\{t\colon s_t\in\mathcal{C}\}$ is an antichain, and if $X=B^E$ it is a branch.
  Hence, the latter estimate together with the definition of the norm in $X$ gives us the desired
  estimate $\|Q\|\leqslant 1$.
\end{proof}

\begin{rem}\label{rem:basic-operators}
  Dualising \Cref{lem:basic-operators} yields $I_{X^*} = B^*Q^*$.  Moreover, observe that for
  finitely supported $x\in S^E$ and all $x^*\in X^*$ we have
  \begin{equation*}
    \langle Bx, x^*\rangle
    = \sum_{t\in\dtree} \langle x, e_t^*\rangle
    \langle e_{s_t}, x^*\rangle
    = \Bigl\langle x,
    \sum_{t\in\dtree} \langle e_{s_t}, x^*\rangle e_t^*
    \Bigr\rangle
    = \langle x, B^* x^* \rangle.
  \end{equation*}
  as well as
  \begin{equation*}
    \langle Qx, x^*\rangle
    = \sum_{t\in\dtree}
    \langle x, e_{s_t}^*\rangle \langle e_t, x^*\rangle
    = \Bigl\langle x,
    \sum_{t\in\dtree} \langle e_t, x^*\rangle e_{s_t}^*
    \Bigr\rangle
    = \langle x, Q^* x^* \rangle,
  \end{equation*}
  We record that
  \begin{equation*}
    B^*x^*
    = \sum_{t\in\dtree} \langle e_{s_t}, x^*\rangle e_t^*
    \qquad\text{and}\qquad
    Q^*x^*
    = \sum_{t\in\dtree} \langle e_t, x^*\rangle e_{s_t}^*
    \qquad (x^*\in X^*).
  \end{equation*}
  By appealing to \Cref{lem:basic-operators}, we have thus established that:
  \begin{itemize}
  \item in either space, $S^E$ or $B^E$, $(e_{s_t})_{t\in\dtree}$ is $1$-equivalent to
    $(e_t)_{t\in\dtree}$,
  \item in either space, $D^E$ or $(B^E)^*$, $(e_{s_t}^*)_{t\in\dtree}$ is $1$-equivalent to
    $(e_t^*)_{t\in\dtree}$.
  \end{itemize}
\end{rem}

\subsection{Subspace annihilation}
\label{sec:subsp-annih-1}
The present short section collects two lemmata describing the weakly null nature of
branches/antichains in $D^E = (S^E)^*$ and the dual space of $B^E$, respectively.
\begin{lem}\label{lem:annihil-1}
  Let $x^*\in D^E$, $t\in \dtree $, and let $\Gamma\subset \dtree$ be a branch.  Then
  \begin{equation*}
    \lim_{n\to\infty} \sup_{s\in\Gamma, |s|\geqslant n}|\langle e_{t^\smallfrown s}, x^*\rangle|
    = 0.
  \end{equation*}
\end{lem}

\begin{proof}
  We can assume that $\|x^*\| = 1$.  Let $\Gamma = \{t_j\colon j\in\mathbb{N}\}$ with
  $\varnothing=t_1\sqsubset t_2\sqsubset t_3\cdots$ and put $s_j=t^\smallfrown t_j$,
  $j\in\mathbb{N}$.  Let $N\in\mathbb{N}$, $(\omega_j)\subset\mathbb{R}^N$ and observe that since
  $|\mathcal{A}\cap\{s_j\colon j\in\mathbb{N}\}|\leqslant 1$ for every antichain $\mathcal{A}$, we
  have
  \begin{equation*}
    \Bigl\langle \sum_{j=1}^N \omega_j e_{s_j}, x^*\Bigr\rangle
    \leqslant \Bigl\|\sum_{j=1}^N \omega_j e_{s_j}\Bigr\|_{S^E}
    = \sup_{\mathcal{A}} \Bigl\|
    \sum_{s\in\mathcal{A}} \Bigl\langle \sum_{j=1}^N \omega_j e_{s_j}, e_s^* \Bigr\rangle e_s
    \Bigr\|_E
    \leqslant \max_{1\leqslant j\leqslant N}|\omega_j|.
  \end{equation*}
  Defining $\omega_j = \sign(\langle e_{s_j}, x^*\rangle)$, the latter estimate yields
  \begin{equation*}
    \sum_{j=1}^N |\langle e_{s_j}, x^*\rangle|
    \leqslant 1
    \qquad (N\in\mathbb{N}).
  \end{equation*}
  Thus, $\lim_j |\langle e_{s_j}, x^*\rangle| = 0$ as claimed.
\end{proof}

\begin{lem}\label{lem:annihil-2}
  Let $x^*\in (B^E)^*$, $t\in \dtree $, and let $\mathcal{A} = \{t_j\colon j\in\mathbb{N}\}$ be an
  antichain.  Suppose that $t_i < t_j$, whenever $i < j$.  Then
  \begin{equation*}
    \lim_{j\to\infty} \langle e_{t_j}, x^*\rangle
    = 0.
  \end{equation*}
\end{lem}

\begin{proof}
  We can assume that $\|x^*\|_{(B^E)^*} = 1$.  Let $N\in\mathbb{N}$, $(\omega_j)\in\mathbb{R}^N$ and
  note that for each branch $\Gamma\in\beta$, we have $|\Gamma\cap\mathcal{A}|\leqslant 1$.  Hence,
  \begin{equation*}
    \Bigl\langle \sum_{j=1}^N \omega_j e_{t_j}, x^*\Bigr\rangle
    \leqslant \Bigl\|\sum_{j=1}^N \omega_j e_{t_j}\Bigr\|_{B^E}
    = \sup_{\Gamma\in\beta} \Bigl\|
    \sum_{s\in\Gamma} \Bigl\langle \sum_{j=1}^N \omega_j e_{t_j}, e_s^* \Bigr\rangle e_s
    \Bigr\|_E
    \leqslant \max_{1\leqslant j\leqslant N}|\omega_j|.
  \end{equation*}
  Next, we define $\omega_j = \sign(\langle e_{t_j}, x^*\rangle)$ and obtain from the latter
  estimate
  \begin{equation*}
    \sum_{j=1}^N |\langle e_{t_j}, x^*\rangle|
    \leqslant 1
    \qquad (N\in\mathbb{N}).
  \end{equation*}
  Thus, $\lim_j |\langle e_{t_j}, x^*\rangle| = 0$ as claimed.
\end{proof}

\subsection{Strategic reproducibility and factorisation}
\label{sec:strat-repr}

For a Banach space $X$ we denote by $\cof(X)$ the set of cofinite dimensional subspaces of $X$,
while $\cof_{w^*}(X^*)$ denotes the set of cofinite dimensionl $w^*$-closed subspaces of $X^*$.
Hereinafter, unless otherwise stated,
\begin{itemize}
\item $X$ is either $S^E$ or $B^E$,
\item $(e_s)_{s\in \dtree}$ denotes the standard Schauder basis in $X$, and
\item $(e_s^*)_{s\in \dtree}$ are the associated biorthogonal functionals.
\end{itemize}
Since in either case $(e_s)_{s\in \dtree}$ is $1$-unconditional, strategic reproducibility
(strategic re\-pro\-du\-ci\-bi\-li\-ty was conceived in~\cite{MR4145794} and investigated further
in~\cite{lechner:motakis:mueller:schlumprecht:2021}) for $(e_s)_{s\in \dtree}$ is reads as follows.
\begin{dfn}\label{dfn:strat-rep-uncond}
  Let $C\geqslant 1$ and consider the following two-player game between player (I) and player (II).
  For $t\in\dtree $, turn $t$ is played out in three steps.
  \begin{itemize}
  \item[Step\! 1:] Player (I) chooses $\eta_t>0$, $W_t\in\mathrm{cof}(X)$, and
    $G_t\in\mathrm{cof}_{w^*}(X^*)$,
  \item[Step\! 2:] Player (II) chooses a finite subset $E_t$ of $\dtree $ and sequences of
    non-negative real numbers $(\lambda_s^{(t)})_{s\in E_t}$, $(\mu_s^{(t)})_{s\in E_t}$ satisfying
    \begin{equation*}
      \sum_{s\in E_t}\lambda_s^{(t)}\mu_s^{(t)} = 1.
    \end{equation*}
  \item[Step\! 3:] Player (I) chooses $(\varepsilon_s^{(t)})_{s\in E_t}$ in $\{-1,1\}^{E_t}$.

  \end{itemize}

  We say that player (II) has a winning strategy in the game $\mathrm{Rep}_{(X,(e_s))}(C)$ if he can
  force the following properties on the result:

  For all $t\in\dtree $ we set
  \begin{equation*}
    b_t = \sum_{s\in E_t}\varepsilon_s^{(t)} \lambda^{(t)}_se_s
    \qquad\text{and}\qquad
    b_t^* = \sum_{s\in E_t}\varepsilon_s^{(t)}\mu^{(t)}_se^*_s
  \end{equation*}
  and demand:
  \begin{enumerate}[(i)]
  \item\label{enu:strat-rep:i} the sequences $(b_t)_{t\in\dtree}$ and $(e_t)_{t\in\dtree}$ are
    impartially $C$-equivalent,
  \item\label{enu:strat-rep:ii} the sequences $(b_t^*)_{t\in\dtree}$ and $(e_t^*)_{t\in\dtree}$ are
    impartially $C$-equivalent,
  \item\label{enu:strat-rep:iii} for all $t\in\mathbb{N}$ we have
    $\mathrm{dist}(b_t, W_t) < \eta_t$, and
  \item\label{enu:strat-rep:iv} for all $t\in\mathbb{N}$ we have
    $\mathrm{dist}(b^*_t, G_t) < \eta_t$.
  \end{enumerate}
  
  We say that $(e_s)_{s\in\dtree}$ is {\em $C$-strategically reproducible in $X$} if for every
  $\eta >0$ player II has a winning strategy in the game $\mathrm{Rep}_{(X,(e_s))}(C+\eta)$.
\end{dfn}

\begin{thm}\label{thm:strat-rep}
  The Schauder basis $(e_t)_{t\in\dtree}$ is $1$-strategically reproducible both in $S^E$ and $B^E$.
\end{thm}  The Schauder basis $(e_t)_{t\in\dtree}$ is $1$-strategically reproducible both in $S^E$ and $B^E$.

\begin{cor}\label{cor:strat-rep-1}
  The system $((e_t,e_t^*)\colon t\in\dtree)$ is strategically supporting and has the
  fac\-to\-ri\-sa\-tion property in both $S^E\times D^E$ and $B^E\times (B^E)^*$.
\end{cor}

\begin{cor}\label{cor:strat-rep-2}
  Suppose that $1\leqslant p,p'\leqslant\infty$ with $1/p + 1/p' = 1$.  Let
  \begin{equation*}
    \langle\cdot,\cdot\rangle\colon \ell^p(X)\times\ell^{p'}(X^*)\to\mathbb{R}
  \end{equation*}
  be the duality bracket given by
  \begin{equation*}
    \langle (x_n)_{n=1}^\infty, (x_n^*)_{n=1}^\infty\rangle = \sum_{n=1}^\infty \langle x_n^*,
    x_n\rangle.
  \end{equation*}
  For each $k\in\mathbb{N}$ let $(e_{k,t})_{t\in\dtree}$ denote a copy of $(e_t)_{t\in\dtree}$ in
  the $k^{\text{th}}$ coordinate of $\ell^p(X)$, and let $(f_{k,t})_{t\in\dtree}$ denote the
  functionals given by $\langle (x_n)_{n=1}^\infty, f_{k,t}\rangle = \langle e_t^*, x_k\rangle$.
  Then the following assertions hold true:
  \begin{romanenumerate}
  \item\label{enu:cor:strat-rep-2:i} $(\ell^p(X),\ell^{p'}(X^*),\langle\cdot,\cdot\rangle)$ is a
    dual pair which satisfies~\eqref{eq:20} with $c=1$;
  \item\label{enu:cor:strat-rep-2:ii} $((e_{k,t},f_{k,t})\colon k\in\mathbb{N},\ t\in\dtree)$ is
    strategically supporting in $\ell^p(X)\times\ell^{p'}(X^*)$;
  \item\label{enu:cor:strat-rep-2:iii} $((e_{k,t},f_{k,t})\colon k\in\mathbb{N},\ t\in\dtree)$ has
    the factorisation property in $\ell^p(X)\times\ell^{p'}(X^*)$.
  \end{romanenumerate}
\end{cor}

\begin{myproof}[Proof of \Cref{thm:strat-rep}]
  We will only present the proof for $S^E$, since the proof for $B^E$ is similar (we use
  \Cref{lem:annihil-2} instead of \Cref{lem:annihil-1}).
  
  Let $t\in\dtree $ and assume we have already played out the turns before $t$ (or none yet if
  $t=\varnothing$).  Let $\tilde{t}$ denote the dyadic predecessor of $t$ and assume that
  $t = \tilde{t}^\smallfrown \alpha$ for some $\alpha\in\{0,1\}$.  We will now describe turn $t$.
  \begin{proofstep}
    Player I chooses $\eta_t > 0$ and the subspaces $W_t\in\cof(S^E)$ and $G_t\in\cof_{w*}(D^E)$.
    Thus, there exist finite sets $V_t\subset D^E$ and $F_t\subset S^E$ such that
    $W_t = (V_t)_\perp$ and $G_t = F_t^\perp$.
  \end{proofstep}
  
  \begin{proofstep}
    We assume that $E_u=\{s_u\}$ ($u < t$) and pick any branch $\Gamma$ through $\alpha$,
    \emph{i.e.}, $s(1) = \alpha$ ($s\in\Gamma$).
    
    First we remark that by \Cref{lem:annihil-1}, $e_s$ converges weakly to $0$, whenever
    $|s|\to\infty$ along the branch $\Gamma$.  Moreover, since $(e_s)_{s\in\dtree}$ is a Schauder
    basis for $S^E$, $e_s^*$ converges to $0$ in the weak* topology, whenever $|s|\to \infty$.
    Thus, we can pick $s_t\in\Gamma$ with $s_t\sqsupset s_{\tilde{t}}$ and $s_t > s_u$ for all
    $u < t$ such that
    \begin{equation*}
      \dist(e_{s_t}, W_t)\leqslant \eta_t
      \qquad\text{and}\qquad
      \dist(e_{s_t}^*, G_t)\leqslant \eta_t.
    \end{equation*}
    Player II chooses $E_t = \{s_t\}$ and $\lambda^{(t)}_{s_t} = \mu^{(t)}_{s_t} = 1$.
  \end{proofstep}

  \begin{proofstep}
    At the end of turn $t$, player I selects a sign $\varepsilon_{s_t}\in\{\pm 1\}$.
  \end{proofstep}\smallskip

  Having completed the game, we \emph{claim} that we have the postulated properties.  By the very
  construction, the properties~\eqref{enu:strat-rep:iii} and~\eqref{enu:strat-rep:iv} of
  \Cref{dfn:strat-rep-uncond} are satisfied.  By \Cref{rem:basic-operators}, the
  properties~\eqref{enu:strat-rep:i} and~\eqref{enu:strat-rep:ii} is satisfied for $C=1$ as well.
\end{myproof}

\begin{proof}[Proof of \Cref{cor:strat-rep-1}]
  By the remark between Definitions~3.9--3.10 in~\cite{MR4145794} and \Cref{thm:strat-rep},
  using~\cite[Theorem~3.12]{MR4145794} yields that $((e_t, e_t^*)\colon t\in\dtree)$ has the
  factorisation property in $S^E$ and in $B^E$.  Let $\mathscr{S}\subset\dtree$ be any subset.  By
  \Cref{lem:subtrees}, either $\mathscr{S}$ or $\dtree\setminus\mathscr{S}$ contains a~linearly
  order-isomorphic subtree $\mathscr{T}$ of $\dtree$.  Thus, by \Cref{rem:basic-operators}, we
  obtain that $(e_s)_{s\in\mathscr{T}}$ is $1$-equivalent to $(e_s)_{s\in\dtree}$ in both $S^E$ and
  $B^E$ and $(e_s^*)_{s\in\mathscr{T}}$ is $1$-equivalent to $(e_s^*)_{s\in\dtree}$ (in $D^E$ and
  $(B^E)^*$).
\end{proof}

In what follows, we shall repeatedly quote the relevant results, in particular~\cite[Co\-ro\-lla\-ry
3.8]{lechner:motakis:mueller:schlumprecht:2021}, which involve the uniform diagonal factorisation
property.  Since the uniform diagonal factorisation property is implied by unconditionality and is
not directly discussed elsewhere in this paper, we recapitulate the relevant facts required for in
\Cref{rem:uni-diag-fac-prop} below.
\begin{rem}\label{rem:uni-diag-fac-prop}
  Note that by the paragraph between Definition~3.9 and Definition~3.10 in~\cite{MR4145794}, a
  $1$-unconditional Schauder basis $(e_n)_{n=1}^\infty$ has the uniform diagonal factorisation
  property; see~\cite[Definition~3.9]{MR4145794} for a definition of the uniform diagonal
  factorisation property.
\end{rem}

\begin{proof}[Proof of \Cref{cor:strat-rep-2}]
  Assertion~\eqref{enu:cor:strat-rep-2:i} is obvious, so we skip the proof.  Since the basis
  $(e_t\colon t\in\dtree)$ is $1$-unconditional, \Cref{rem:uni-diag-fac-prop} yields that
  $(e_t\colon t\in\dtree)$ has the uniform diagonal factorisation property.  In \Cref{thm:strat-rep}
  we already proved that $(e_t\colon t\in\dtree)$ is also $1$-strategically reproducible in $X$.
  Hence, using~\cite[Theorem~7.6]{MR4145794} for $1\leqslant p < \infty$
  and~\cite[Corollary~3.8]{lechner:motakis:mueller:schlumprecht:2021} for $p = \infty$, we
  obtain~\eqref{enu:cor:strat-rep-2:iii}.\smallskip

  Now we only need to show that~\eqref{enu:cor:strat-rep-2:ii} is true as well.  To this end, let
  $\mathscr{S}\subset\mathbb{N}\times\dtree$ be any subset.  For each $k\in\mathbb{N}$, define
  $\mathscr{S}_k = \{t\in\dtree\colon (k,t)\in\mathscr{S}\}$ and note that by \Cref{lem:subtrees},
  either $\mathscr{S}_k$ or $\dtree\setminus\mathscr{S}_k$ contains a~linearly order-isomorphic
  subtree.  Thus, if we define
  \begin{equation*}
    \mathscr{K}
    = \{k\in\mathbb{N}\colon \text{$\mathscr{S}_k$ contains a~linearly order-isomorphic subtree}\},
  \end{equation*}
  then either $\mathscr{K}$ or $\mathbb{N}\setminus\mathscr{K}$ is infinite.  Let us assume without
  restriction that $\mathscr{K}$ is infinite and for each $k\in\mathscr{K}$ let $\mathscr{T}_k$
  denote a~linearly order-isomorphic subtree.  For fixed $k\in\mathbb{N}$,
  \Cref{rem:basic-operators} asserts that $(e_{k,t}\colon t\in\mathscr{T}_k)$ is $1$-equivalent to
  $(e_{k,t}\colon t\in\dtree)$ in $X$ as well as that $(f_{k,t}\colon t\in\mathscr{T}_k)$ is
  $1$-equivalent to $(f_{k,t}\colon t\in\dtree)$ in $X^*$.  Hence,
  \begin{itemize}
  \item $(e_{k,t}\colon k\in\mathscr{K},\ t\in\mathscr{T}_k)$ is $1$-equivalent to
    $(e_{k,t}\colon k\in\mathbb{N},\ t\in\dtree)$ in $\ell^p(X)$ and
  \item $(f_{k,t}\colon k\in\mathscr{K},\ t\in\mathscr{T}_k)$ is $1$-equivalent to
    $(f_{k,t}\colon k\in\mathbb{N},\ t\in\dtree)$ in $\ell^{p'}(X^*)$.
  \end{itemize}
  This shows that $((e_{k,t},f_{k,t})\colon k\in\mathbb{N},\ t\in\dtree)$ is strategically
  supporting in $\ell^p(X)\times\ell^{p'}(X^*)$ as claimed.
\end{proof}

\section{Factorisation of the identity in \texorpdfstring{$D^E$}{D^E}}
\label{sec:factorisation-d=s1}
The present section is a step towards the proof of \Cref{th:c} as it specialises to the problem of
factorisation of operators on the space $D^E$, when $E$ is a Banach space with a $1$-subsymmetric
Schauder basis that is incomparably non-$c_0$ on antichains. \smallskip

If not stated otherwise, $(e_s)_{s\in\dtree}$ denotes the standard Schauder basis for $S^E$ and
$(f_s)_{s\in\dtree}$ the biorthogonal functionals.  Note that $(f_s)_{s\in\dtree}$ forms a weak*
Schauder basis for $D^E$.

\subsection{Subspace annihilation}
\label{sec:subsp-annih}
Suppose that $(e_s)_{s\in\dtree}$ is a $1$-subsymmetric Schauder basis for the Banach space $E$.  We
say that $(e_s)_{s\in\dtree}$ is \emph{incomparably non-$c_0$ on antichains}, whenever for all sets
$\sigma_j\subset\dtree$ with $\sigma_j\perp\sigma_k$, $j\neq k\in\mathbb{N}$, we have that
\begin{equation}\label{eq:incomp-non-l1}
  \inf_{\substack{N\in\mathbb{N}\\\|(a_j)\|_{\ell_N^1} = 1}} \sup\Bigl\{
  \sum_{j=1}^N |a_j|\|z_j\|_{E}\colon \text{$\mathcal{A}$ antichain,
    $\supp(z_j)\subset\sigma_j\cap\mathcal{A}$, $\Bigl\|\sum_{j=1}^N z_j\Bigr\|_E = 1$}
  \Bigr\}
  = 0.
\end{equation}

\begin{rem}\label{rem:incomp-non-l1}
  Note that $(e_s)_{s\in\dtree}$ is incomparably non-$c_0$ on antichains if for some $r < \infty$,
  every sequence of incomparable sets $(\sigma_j)_{j=1}^\infty$ and every finite sequence of vectors
  $(z_j)_{j=1}^N$ with $z_j\in E$, $\supp(z_j)\subset\sigma_j$ satisfies a lower $r$-estimate,
  \emph{i.e.},
  \begin{equation}\label{eq:lower-r-estimate}
    \Bigl\|\sum_{j=1}^N z_j\Bigr\|_E
    \geqslant c_r \Bigl(\sum_{j=1}^N\|z_j\|_E^r\Bigr)^{1/r},
  \end{equation}
  where the constant $c_r$ neither does depend on $N$ nor on $(z_j)_{j=1}^N$.

  To see this, define $a_j = N^{-1}$ and let $z_j$ with $\supp(z_j)\subset\sigma_j$, and
  $\Bigl\|\sum_{j=1}^N z_j\Bigr\|_E = 1$.  Then by~\eqref{eq:lower-r-estimate}, we obtain
  \begin{equation*}
    \Bigl(\sum_{j=1}^N\|z_j\|_E^r\Bigr)^{1/r}
    \leqslant c_r^{-1},
  \end{equation*}
  and thus,
  \begin{equation*}
    \sum_{j=1}^N |a_j|\|z_j\|_{E}
    \leqslant N^{-1+1/r'} \Bigl(\sum_{j=1}^N \|z_j\|_{E}^r\Bigr)^{1/r}
    \leqslant c_r^{-1} N^{-1+1/r'}.
  \end{equation*}
  Clearly, the right hand side tends to $0$ if $N\to\infty$.
\end{rem}

\begin{lem}\label{lem:non-l1-tree-splicing}
  Let $T_1,\ldots,T_K$ be subtrees of $\dtree $ with $T_j\perp T_k$, $j\neq k$ and suppose that the
  $1$-subsymmetric Schauder basis of $E$ is incomparably non-$c_0$ on antichains.  Then for all
  $\eta > 0$ and every operator $A\colon D^E\to D^E$ and $y\in S^E$, there exist subtrees
  $S_1\ldots,S_K$ with $S_k\subset T_k$, $1\leqslant k\leqslant K$ such that
  \begin{equation*}
    \sup_{\|x\|_{D^E}\leqslant 1} \bigl|\big\langle y, A (x|_S)\rangle\bigr|
    \leqslant \eta,
  \end{equation*}
  where $S=\bigcup_{k=1}^K S_k$, $x|_S = \sum_{s\in S} \langle e_s, x\rangle f_s$ and the series
  converges in the weak* topology of $D^E$.
\end{lem}

\begin{proof}
  Now let $\eta > 0$, $A\colon D^E\to D^E$, $y\in S^E$, and assume the assertion is false, i.e., for
  all subtrees $S_1,\ldots S_K$ with $S_k\subset T_k$, $1\leqslant k\leqslant K$ we have that
  \begin{equation}\label{eq:2}
    \sup_{\|x\|_{D^E}\leqslant 1} \bigl|\big\langle y, A (x|_S)\rangle\bigr|
    > \eta,
  \end{equation}
  where $S = \bigcup_{k=1}^K S_k$.
  
  Pick infinite antichains $\mathcal{A}_k$ in $T_k$, $1\leqslant k\leqslant K$ and let
  $\mathcal{A}_k = \{t_k^j\}_{j=1}^\infty$.  Define the subtrees
  $S_k^j = \{s\in T_k\colon s\sqsupseteq t_k^j\}$, $j\in\mathbb{N}$ of $T_k$,
  $1\leqslant k\leqslant K$ and note that $S_k^j\perp S_{k'}^{j'}$ for all $(j,k)\neq (j',k')$.  In
  particular, if we set $S^j=\bigcup_{k=1}^K S_k^j$, $j\in\mathbb{N}$, then $S^j\perp S^{j'}$,
  $j\neq j'\in\mathbb{N}$.  Now pick $N\in\mathbb{N}$ and $(a_j)_{j=1}^N$ according
  to~\eqref{eq:incomp-non-l1} such that
  \begin{equation}\label{eq:28}
    \sup\Bigl\{
    \sum_{j=1}^N |a_j|\|z_j\|_{E}\colon \text{$\mathcal{A}$ antichain,
      $\supp(z_j)\subset S^j\cap\mathcal{A}$, $\Bigl\|\sum_{j=1}^N z_j\Bigr\|_E = 1$}
    \Bigr\}
    \leqslant \frac{\eta}{2 \|A\| \|y\|_{S^E}}.
  \end{equation}
  Now, for each $1\leqslant j\leqslant N$, we use~\eqref{eq:2} on $S_1^j,\ldots S_K^j$ to find a
  vector $x_j\in D^E$ with $\|x_j\|=1$, $\supp(x_j)\subset S^j$ such that
  $ \langle A x_j, y\rangle > \eta$.
  
  Multiplying with $|a_j|$ and summing over $1\leqslant j\leqslant N$ yields
  \begin{equation}\label{eq:29}
    \eta
    < \Bigl\langle A \sum_{j=1}^N |a_j| x_j, y\Bigr\rangle
    \leqslant \|A\| \Bigl\|\sum_{j=1}^N |a_j| x_j\Bigr\|_{D^E} \|y\|_{S^E}.
  \end{equation}
  We will now estimate $\Bigl\|\sum_{j=1}^N |a_j| x_j\Bigr\|_{D^E}$.  Observe that
  \begin{align*}
    \Bigl\|\sum_{j=1}^N |a_j| x_j\Bigr\|_{D^E}
    &= \sup_{(c_s)_{s\in\dtree}}
      \frac{\sum_{j=1}^N |a_j| \langle\sum_{s\in S^j} c_s e_s ,x_j\rangle}
      {\Bigl\|\sum_{j=1}^N\sum_{s\in S^j} c_s e_s\Bigr\|_{S^E}}
      \leqslant \sup_{(c_s)_{s\in\dtree}}
      \frac{\sum_{j=1}^N |a_j| \bigl\|\sum_{s\in S^j} c_s e_s\bigr\|_{S^E}}
      {\Bigl\|\sum_{j=1}^N\sum_{s\in S^j} c_s e_s\Bigr\|_{S^E}}\\
    &= \sup_{(c_s)_{s\in\dtree}}
      \frac{\sum_{j=1}^N |a_j| \sup_{\mathcal{A}} \bigl\|\sum_{s\in S^j\cap\mathcal{A}} c_s e_s\bigr\|_E}
      {\Bigl\|\sum_{j=1}^N\sum_{s\in S^j} c_s e_s\Bigr\|_{S^E}},
  \end{align*}
  where the supremum over the $(c_s)_{s\in\dtree}$ is restricted to
  $\bigl\|\sum_{j=1}^N\sum_{s\in S^j} c_s e_s\bigr\|_E\neq 0$ and the other supremum is taken over
  all antichains $\mathcal{A}$.  Since $S^j\perp S^{j'}$, $j\neq j'$, the antichains $\mathcal{A}$
  which depend on $j$ have a common antichain, \emph{i.e.},
  \begin{equation*}
    \sum_{j=1}^N |a_j| \sup_{\mathcal{A}} \Bigl\|\sum_{s\in S^j\cap\mathcal{A}} c_s e_s\Bigr\|_E
    = \sup_{\mathcal{A}} \sum_{j=1}^N |a_j| \Bigl\|\sum_{s\in S^j\cap\mathcal{A}} c_s e_s\Bigr\|_E.
  \end{equation*}
  With this observation and the definition of the norm in $S^E$, we obtain
  \begin{align*}
    \Bigl\|\sum_{j=1}^N |a_j| x_j\Bigr\|_{D^E}
    &\leqslant \sup_{(c_s)_{s\in\dtree}} \sup_{\mathcal{A}}
      \frac{\sum_{j=1}^N |a_j| \bigl\|\sum_{s\in S^j\cap\mathcal{A}} c_s e_s\bigr\|_E}
      {\Bigl\|\sum_{j=1}^N\sum_{s\in S^j} c_s e_s\Bigr\|_{S^E}}\\
    &= \sup_{(c_s)_{s\in\dtree}} \sup_{\mathcal{A}} \inf_{\mathcal{B}}
      \frac{\sum_{j=1}^N |a_j| \bigl\|\sum_{s\in S^j\cap\mathcal{A}} c_s e_s\bigr\|_E}
      {\bigl\|\sum_{j=1}^N\sum_{s\in S^j\cap\mathcal{B}} c_s e_s\bigr\|_E},
  \end{align*}
  where the infimum is taken over all antichains $\mathcal{B}$.  In particular, choosing
  $\mathcal{B} = \mathcal{A}$ and using~\eqref{eq:28} yields
  \begin{align*}
    \Bigl\|\sum_{j=1}^N |a_j| x_j\Bigr\|_{D^E}
    &\leqslant \sup_{(c_s)_{s\in\dtree}} \sup_{\mathcal{A}}
      \frac{\sum_{j=1}^N |a_j| \bigl\|\sum_{s\in S^j\cap\mathcal{A}} c_s e_s\bigr\|_E}
      {\Bigl\|\sum_{j=1}^N\sum_{s\in S^j\cap\mathcal{\mathcal{A}}} c_s e_s\Bigr\|_E}\\
    &= \sup\biggl\{ 
      \sum_{j=1}^N |a_j| \|z_j\|:
      \text{$\mathcal{A}$ antichain,
      $\supp(z_j)\subset S^j\cap\mathcal{A}$, $\Bigl\|\sum_{j=1}^N z_j\Bigr\|_E = 1$}
      \biggr\}\\
    &\leqslant \frac{\eta}{2 \|A\| \|y\|_{S^E}}.
  \end{align*}
  Inserting the latter estimate into~\eqref{eq:29} leads to a contradiction.
\end{proof}

\begin{lem}\label{lem:rademacher}
  Let $x_1,\ldots,x_n\in D^E$, let $\mathscr{T}$ denote a subtree of $\dtree $ and let $\eta > 0$.
  Then there exists an $s\in \mathscr{T}$ such that
  \begin{equation*}
    \max_{1\leqslant j\leqslant n}|\langle e_s, x_j\rangle|\leqslant \eta.
  \end{equation*}
\end{lem}

\begin{proof}
  For fixed $1\leqslant j\leqslant n$ and $\Gamma\in\beta$, define
  $y = \sum_{s\in\Gamma} \varepsilon_s e_s$, where $\varepsilon_s = \sign(\langle e_s, x_j\rangle)$,
  $s\in\Gamma$, and observe that since every antichain intersects the branch $\Gamma$ at most once,
  we have
  \begin{equation*}
    \|y\|_{S^E}
    = \sup_{\mathcal{A}} \Bigl\|\sum_{s\in\mathcal{A}\cap\Gamma} \langle y, f_s\rangle e_s\Bigr\|_E
    = \sup_{\mathcal{A}} \Bigl\|\sum_{s\in\mathcal{A}\cap\Gamma} \varepsilon_s e_s\Bigr\|_E
    = 1.
  \end{equation*}
  Thus, we obtain
  \begin{equation*}
    \infty
    > \|x_j\|_{D^E}
    \geqslant \langle y, x_j\rangle
    = \sum_{s\in\Gamma} |\langle e_s, x_j\rangle|
    \geqslant \sum_{s\in\Gamma\cap \mathscr{T}} |\langle e_s, x_j\rangle|,
  \end{equation*}
  hence $|\langle e_s, x_j\rangle|\to 0$ as $s$ tends to infinity along the branch
  $\Gamma\cap \mathscr{T}$ of the subtree $\mathscr{T}$.  Since there are only finitely many $x_j$,
  the assertion follows.
\end{proof}

\subsection{Factorisation of the identity}
\label{sec:factorisation}
In the present section we fix a Banach space $E$ with a $1$-subsymmetric Schauder basis for the
Banach space.  Let $(e_s)_{s\in\dtree}$ denote the standard unit vector basis in $S^E$ and let
$(f_s)_{s\in\dtree}$ denote the associated biorthogonal functionals in $D^E = (S^E)^*$.

\begin{thm}\label{thm:factor-D}
  Suppose that $E$ is incomparably non-$c_0$ on antichains.  Let $T\colon D^E\to D^E$ be an operator
  having large diagonal with respect to $(f_s)_{s\in\dtree}$, {i.e.},
  \begin{equation}\label{eq:3}
    \delta = \inf_{s\in \dtree} |\langle e_s, Tf_s\rangle|
    > 0.
  \end{equation}
  Then for each $\eta > 0$ there exist operators $A,B\colon D^E\to D^E$ such that $ATB = I_{D^E}$
  and $\|A\|\|B\|\leqslant \frac{1+\eta}{\delta}$.
\end{thm}

\begin{cor}\label{cor:factor-D}
  Suppose that $(e_s)_{s\in\dtree}$ is incomparably non-$c_0$ on antichains.  Then:
  \begin{itemize}
  \item $((f_s,e_s)\colon s\in\dtree)$ is strategically supporting and
  \item $((f_s,e_s)\colon s\in\dtree)$ has the factorisation property in $D^E\times S^E$.
  \end{itemize}
\end{cor}

\begin{proof}[Proof of \Cref{cor:factor-D}]
  By \Cref{thm:factor-D}, $((f_s,e_s)\colon s\in\dtree)$ has the factorisation property.  To see
  that $((f_s,e_s)\colon s\in\dtree)$ is also strategically supporting, we reason as in the proof of
  \Cref{cor:strat-rep-1}.
\end{proof}

\begin{myproof}[Proof of \Cref{thm:factor-D}]
  We define the constant $\eta_0=\eta_0(\delta, \eta)$ such that
  \begin{equation}
    \label{eq:14}
    \frac{\eta_0}{3\delta}
    < 1
    \qquad\text{and}\qquad
    \frac{1}{1-\frac{\eta_0}{3\delta}}
    \leqslant 1 + \eta.
  \end{equation}

  First, note that since the weak* Schauder basis $(f_s)_{s\in\dtree}$ is $1$-unconditional, we may
  assume that $\delta = \inf_{s\in \dtree} \langle Tf_s, e_s\rangle > 0$.  In this proof, we will
  regularly identify $b_t$ and $b_t^*$ with $b_n$ and $b_n^*$, where $n$ is the index of the node
  $s$ in the standard linear order of the tree $\dtree $.
  \begin{proofstep}[Diagonalisation of the operator]
    We will now inductively define biorthogonal subsequences $(b_t)_{t\in \dtree}$ and
    $(b_t^*)_{t\in \dtree}$ of $(f_s)_{s\in \dtree}$ and $(e_s)_{s\in \dtree}$.  We begin our
    construction by putting $b_\varnothing=f_\varnothing$ and $b_\varnothing^*=e_\varnothing$.  We
    use \Cref{lem:non-l1-tree-splicing} on the subtrees
    $S^\alpha = \{s\in \dtree\colon s\sqsupseteq \alpha\}$, $\alpha\in\{0,1\}$ and find subtrees
    $S_\varnothing^\alpha\subset S^\alpha$, $\alpha\in\{0,1\}$ such that
    \begin{equation}
      \label{eq:5}
      \sup_{\|x\|_{D^E}\leqslant 1} |\langle b_\varnothing^*, T (x|_{S_\varnothing})\rangle|\leqslant \eta_0 / 4,
    \end{equation}
    where $S_\varnothing = S_\varnothing^0\cup S_\varnothing^1$.

    Let $t_0\in \dtree $ with $t_0 > \varnothing$ and assume that we have selected finite unions of
    pairwise incomparable subtrees $S_t$, $t < t_0$ with $S_t\supset S_{t+1}$ and constructed $b_t$,
    $b_t^*$, $t < t_0$ such that
    \begin{subequations}\label{eq:6}
      \begin{align}\label{eq:6:a}
        b_t
        = f_{s_t},
        \quad b_t^*
        = e_{s_t}\qquad (t < t_0),
      \end{align}
      where
      \begin{align}\label{eq:6:b}
        s_t
        &\in S_{t-1}
          \quad (\varnothing < t < t_0),
        &s_{t^\smallfrown\alpha}
        &\sqsubset s_{t},
          \quad (\alpha\in\{0,1\},\ t < t_0-1),
      \end{align}
      $s_{t^\smallfrown\alpha}$ is to the left of $s_t$ if $\alpha = 0$ and to the right of $s_t$ if
      $\alpha = 1$,
      \begin{equation}
        \label{eq:6:bb}
        \text{$\{s\in S_{t_0-1}\colon s\sqsupset s_{t}\}$ is a union of at least two incomparable subtrees}
      \end{equation}
      whenever $t < t_0$, as well as
      \begin{align}
        \sum_{t_1 < t} |\langle b_{t}^*, T b_{t_1}\rangle|
        & \leqslant \eta_0 4^{-\mathcal{O}(t)}
          \qquad (t < t_0),
          \label{eq:6:c}\\
        \sup_{\|x\|_{D^E}\leqslant 1} |\langle b_{t}^*, T (x|_{S_{t}})\rangle|
        &\leqslant \eta_0 4^{-\mathcal{O}(t)}
          \qquad (t < t_0).
          \label{eq:6:d}
      \end{align}
    \end{subequations}
    We will now construct a finite union of pairwise incomparable subtrees
    $S_{t_0}\subset S_{t_{0}-1}$ and $b_{t_0}$, $b_{t_0}^*$ such that~\eqref{eq:6} is satisfied for
    all $t\leqslant t_0$.  Pick $\alpha\in\{0,1\}$ such that
    $(\widetilde{t_0})^\smallfrown\alpha = t_0$.  Since
    $S = \{s\in S_{t_0-1}\colon s\sqsupset s_{\widetilde{t_0}}\}$ is a union of at least two
    pairwise incomparable subtrees $S^0,S^1\subset S$ such that $S^0$ is to the left of $S^1$.  We
    use \Cref{lem:rademacher} with $x_t=T b_t$, $t < t_0$ and the subtree $S^\alpha$ to find a node
    $s_{t_0}\in S^\alpha$ such that
    \begin{equation}
      \label{eq:7}
      \sum_{t<t_0} |\langle e_{s_{t_0}}, T b_t\rangle|
      \leqslant \eta_0 4^{-\mathcal{O}(t_0)}.
    \end{equation}
    We put $b_{t_0}=f_{s_{t_0}}$ and $b_{t_0}^* = e_{s_{t_0}}$ and note that~\eqref{eq:6:a}
    and~\eqref{eq:6:b} are both satisfied for all $t\leqslant t_0$.  Next, we use
    \Cref{lem:non-l1-tree-splicing} on the finite union of pairwise incomparable subtrees
    $S':=S_{t_0-1}\setminus S^\alpha\cup\{s\in S^\alpha\colon s\sqsupset s_{t_0}\}\subset S_{t_0-1}$
    to obtain a union of pairwise incomparable subtrees $S_{t_0}\subset S'$ such that
    \begin{equation}
      \label{eq:8}
      \sup_{\|x\|_{D^E}\leqslant 1} |\langle b_{t_0}^*, T (x|_{S_{t_0}})\rangle|
      \leqslant \eta_0 4^{\mathcal{O}(t_0)},
    \end{equation}
    Combining \eqref{eq:7} with~\eqref{eq:8} shows that~\eqref{eq:6:c} and~\eqref{eq:6:d} both hold
    true for all $t\leqslant t_0$.  Finally, we observe that
    $\{s\in S^\alpha\colon s\sqsupset s_{t_0}\}$ is the union of two incomparable subtrees and that
    the application of \Cref{lem:non-l1-tree-splicing} replaced each subtree in $S'$ with another
    subtree; hence, \eqref{eq:6:bb} is satisfied for all $t\leqslant t_0$, as well.  This concludes
    the inductive construction.
  \end{proofstep}

  \begin{proofstep}[Conclusion of the proof]
    It is clear from the principle of our construction that $\{s_t\colon t\in \dtree \}$ is a
    subtree.  By \Cref{rem:basic-operators}, the operators $B,Q\colon D^E\to D^E$ given by
    \begin{equation*}
      Bx
      = \sum_{t\in \dtree} \langle e_t, x\rangle b_t
      \qquad\text{and}\qquad
      Qx
      = \sum_{t\in \dtree} \langle b_t^*, x\rangle f_t
      \qquad (x\in D^E)
    \end{equation*}
    satisfy $\|B\|=\|Q\|=1$.  Define the norm-one projection $P\colon D^E\to D^E$ by $P=BQ$ and put
    $Z = P(D^E)$, \emph{i.e.},
    \begin{equation}
      \label{eq:9}
      Z = \Big\{z = \sum_{t\in \dtree} a_t b_t\colon 
      a_t\in \mathbb R, \|z\|_{D^E} < \infty\Big\},
    \end{equation}
    where the series converges in the weak* topology of $D^E$.  The following diagram commutes:
    \begin{equation}\label{eq:15}
      \vcxymatrix{D^E \ar[r]^{I_{D^E}} \ar[d]_{B} & D^E\\
        Z \ar[r]_{I_{D^E}} & Z \ar[u]_{Q}}
      \qquad \|B\|,\|Q\| = 1.
    \end{equation}
    Next, define $U\colon D^E\to Z$ by
    \begin{equation}
      \label{eq:10}
      Ux
      = \sum_{t\in \dtree}
      \frac{\langle b_t^*, x\rangle}{\langle b_t^*, T b_t\rangle} b_t
    \end{equation}
    and note that by~\eqref{eq:3} and $1$-unconditionality of $(f_t)_{t\in\dtree}$, we have that
    \begin{equation}
      \label{eq:11}
      \|Ux\|_{D^E}
      \leqslant \frac{1}{\delta}\Bigl\|
      \sum_{t\in \dtree} \langle b_t^*, x\rangle b_t
      \Bigr\|_{D^E}
      =\frac{1}{\delta} \|BQx\|_{D^E}
      =\frac{1}{\delta} \|Px\|_{D^E}
      \leqslant\frac{1}{\delta} \|x\|_{D^E}
      \qquad (x\in D^E).
    \end{equation}
    Let $z = \sum_{i=1}^\infty a_i b_i\in Z$ (recall that we identify a node with its index in the
    standard linear order of the tree) and observe that
    \begin{equation}\label{eq:12}
      UTz - z
      = \sum_{i=1}^\infty \Bigl(
      \sum_{j\colon j < i} a_j
      \frac{\langle b_i^*, T b_j\rangle}{\langle b_i^*, Tb_i\rangle}
      + \frac{\big\langle b_i^*, T \sum_{j\colon j > i} a_j b_j\big\rangle}{\langle b_i^*, Tb_i\rangle}
      \Bigr) b_i.
    \end{equation}
    Using~\eqref{eq:6}, $|a_j|\leqslant \|z\|$, $j\in\mathbb{N}$ and~\eqref{eq:3}, we obtain
    \begin{align*}
      \|UTz - z\|_{D^E}
      &\leqslant \sum_{i=1}^\infty
        \sum_{j\colon j < i} |a_j|
        \frac{|\langle b_i^*, T b_j\rangle|}{|\langle b_i^*, Tb_i\rangle|}
        + \frac{|\big\langle b_i^*, T \sum_{j\colon j > i} a_j b_j\big\rangle|}{|\langle b_i^*, Tb_i\rangle|}\\
      &\leqslant \frac{\eta_0}{3\delta}\|z\|_{D^E}.
    \end{align*}
    Let $J\colon Z\to D^E$ be the formal inclusion map, \emph{i.e.}, $Jz=z$ ($z\in Z$).  Let us
    define the operator $V\colon D^E\to Z$ by $V=(UTJ)^{-1}U$.  By \eqref{eq:14}, $V$ is well
    defined and the following diagram commutes:
    \begin{equation}\label{eq:16}
      \vcxymatrix{
        Z \ar[rr]^{I_Z} \ar@/_/[dd]_J \ar[rd]_{UTJ} & & Z\\
        & Z \ar[ru]^{(UTJ)^{-1}} &\\
        D^E \ar[rr]_T & & D^E \ar[lu]_U \ar[uu]_V
      }
      \qquad \|J\|\|V\| \leqslant (1+\eta)/\delta.
    \end{equation}
    Merging the diagrams~\eqref{eq:15} and~\eqref{eq:16} concludes the proof.\qedhere
  \end{proofstep}
\end{myproof}

\section{Further applications in classical Banach spaces}
\label{sec:applications}
Even though we have been primarily interested in the class of stopping-time space, the methods
developed along the way to prove that for various spaces $X$ in this class, the set described by
\eqref{eq:Mx} is the unique maximal ideal of $\mathscr{B}(X)$, are general enough and apply to other
Banach sequence/function spaces.  In the present section for a given class of spaces, having
verified that \Cref{thm:max-ideal:2} applies, we derive the uniqueness of the maximal ideal of
$\mathscr{B}(X)$.

\subsection{Application to \pmb{$L^p$}, Hardy spaces, \pmb{$\bmo$}, and \pmb{$\mathrm{SL^\infty}$}}
\label{sec:appl-hardy-spac}

The collection of \emph{dyadic intervals $\mathcal{D}$} is given by
\begin{equation*}
  \mathcal{D}
  = \{ [(k-1)2^{-n}, k2^{-n})\colon 1\leqslant k\leqslant 2^n, n\geqslant 0\}.
\end{equation*}
Given any $\mathcal{C}\subset\mathcal{D}$, we define
\begin{equation*}
  \limsup\mathcal{C}
  = \{t\in[0,1]\colon \text{$t$ is contained in infinitely many $I\in\mathcal{C}$}\}.
\end{equation*}

For $I\in\mathcal{D}$, the $L^\infty$-normalised \emph{Haar function $h_I$} is given by
$h_I = \mathds{1}_{I_0} - \mathds{1}_{I_1}$, where $I_0\in\mathcal{D}$ denotes the left half of $I$,
$I_1\in\mathcal{D}$ denotes the right half of $I$ and $\mathds{1}_A$ is the indicator function of
the set $A$.  The sequence $(h_I\colon I\in\mathcal{D})$ is called \emph{the Haar system}.  Since the
dyadic intervals form a dyadic tree in the sense of \Cref{sec:dyadic-tree}, the notion of the
standard linear order of a dyadic tree applies to $\mathcal{D}$, which ultimately linearly orders
the Haar system (the standard linear order of the Haar system).  In this standard linear order, the
Haar system is a Schauder basis in $L^p$, $1\leqslant p < \infty$; for the parameters
$1 < p < \infty$, the Haar system is even unconditional (see \cite{marcinkiewicz:1937}~and~\cite{MR1576148}).\smallskip

The dyadic \emph{Hardy space $H^1$} is given by
\begin{equation*}
  \Bigl\{ f\in L^1\colon \int_0^1 f(t)\,{\rm d}t = 0,\ \|f\|_{H^1} < \infty \Bigr\},
\end{equation*}
where the square function norm $\|\cdot\|_{H^1}$ is given by
\begin{equation*}
  \Bigl\|\sum_{I\in\mathcal{D}} a_I h_I\Bigr\|_{H^1}
  = \Bigl\|\Bigl(
  \sum_{I\in\mathcal{D}} a_I^2 h_I^2
  \Bigr)^{1/2}\Bigr\|_{L^1}
  = \int_0^1\Bigl(\sum_{I\in\mathcal{D}} a_I^2 h_I^2(t)\Bigr)^{1/2}\, \mathrm{d}t.
\end{equation*}
We note that the Haar system is a $1$-unconditional Schauder basis for $H^1$ (see \cite[Section
6]{zbMATH02020174}).  The (non-separable) dual of $H^1$ is denoted by $\bmo$.  Hence, the Haar
system forms a $1$-unconditional weak* Schauder basis in $\bmo$.  A closely related space is the
also non-separable Banach space $\mathrm{SL^\infty}$ (see~\cite{jones:mueller:2004}), which is given
by
\begin{equation*}
  \mathrm{SL^\infty}
  = \Bigl\{ f\in L^2\colon \int_0^1 f(t)\, {\rm d}t = 0,\ \|f\|_{\mathrm{SL^\infty}} < \infty \Bigr\},
\end{equation*}
where the norm $\|\cdot\|_{\mathrm{SL^\infty}}$ is given by
\begin{equation*}
  \Big\| \sum_{I\in\mathcal{D}} a_I h_I \Big\|_{\mathrm{SL^\infty}}
  = \Bigl\| \Bigl(\sum_{I\in\mathcal{D}} a_I^2 h_I^2\Bigr)^{1/2} \Bigr\|_{L^\infty}
  = \esssup_{t\in [0,1]} \Bigl(\sum_{I\in\mathcal{D}} a_I^2 h_I^2(t)\Bigr)^{1/2}.
\end{equation*}

\begin{thm}\label{thm:max-ideals-Hp-bmo-SLinfty}
  The following systems are strategically supporting and have the positive factorisation property:
  \begin{romanenumerate}
  \item $((h_I/|I|^{1/p},h_I/|I|^{1/p'})\colon I\in\mathcal{D})$ in $L^p\times L^{p'}$, where
    $1\leqslant p < \infty$ and $1/p + 1/p' = 1$.
  \item $((h_I/|I|,h_I)\colon I\in\mathcal{D})$ in $H^1\times\bmo$.
  \item $((h_I,h_I^*)\colon I\in\mathcal{D})$ in $H^1\times Y$, where $Y$ is the norm-closure of the
    biorthogonal functionals $h_I^*$ in $(\mathrm{SL^\infty})^*$.
  \end{romanenumerate}
  Moreover, \eqref{eq:20} is satisfied with $c=1$ in all above cases.
\end{thm}

Before we prove \Cref{thm:max-ideals-Hp-bmo-SLinfty}, we record an immediate consequence of
\Cref{thm:max-ideals-Hp-bmo-SLinfty} and \Cref{thm:max-ideal:2}.
\begin{cor}\label{cor:max-ideals-Hp-bmo-SLinfty}
  Let $X$ denote one of the spaces $L^p$ $(1\leqslant p < \infty)$, $H^1$, $\bmo$, or
  $\mathrm{SL^\infty}$.  Then $\mathscr{M}_X$ is the unique closed proper maximal ideal of
  $\mathscr{B}(X)$.
\end{cor}
We would like to point out that the results stated in \Cref{cor:max-ideals-Hp-bmo-SLinfty} also
follow by combining~\cite[Section~5]{dosev:johnson:2010} with~\cite{maurey:sous:1975} for $L^p$,
with~\cite{mueller:1988} for $H^1$ and $\bmo$ and with~\cite{lechner:2018:factor-SL} for
$\mathrm{SL^\infty}$.

\begin{myproof}[Proof of \Cref{thm:max-ideals-Hp-bmo-SLinfty}]
  Note that only in the case $X=\mathrm{SL^\infty}$, we need to be diligent in regards
  to~\eqref{eq:20}.  In the subsequent proofs, we will several times assert that a pair of bases has
  the `(positive) factorisation property'.  By that we wish to indicate that within the respective
  paper the assertion is not stated explicitly for the positive factorisation property, but for the
  factorisation property.  By \Cref{rem:factorisation-property}, the factorisation property implies
  the positive factorisation property.

  \begin{proofcase}[Case $X = L^p$ $(1 < p < \infty)$]
    For the Hölder conjugates $1 < p,p' < \infty$, Andrew~\cite{andrew:1979} (we remark that the
    result is not stated explicitly but directly contained within the proof; see
    also~\cite{MR4145794} for an exposition) showed that
    $((h_I/|I|^{1/p},h_I/|I|^{1/p'})\colon I\in\mathcal{D})$ has the (positive) factorisation
    property in $L^p\times L^{p'}$.  Gamlen and Gaudet~\cite{gamlen:gaudet:1973} proved that for any
    collection $\mathcal{C}\subset\mathcal{D}$ with $|\limsup\mathcal{C}| > 0$ (here, $|\cdot|$
    denotes the Lebesgue-measure), there is a block basis $(b_I\colon I\in\mathcal{D})$ of
    $(h_{I}\colon I\in\mathcal{C})$ that is equivalent to $(h_I\colon I\in\mathcal{D})$ in $L^p$
    (and in $L^{p'}$).  Reviewing their proof, we find that $(b_I\colon I\in\mathcal{D})$ can be
    constructed so that it also satisfies~\eqref{enu:dfn:strat-supp:iii}
    and~\eqref{enu:dfn:strat-supp:iv} in \Cref{dfn:strat-supp}.  Since for any collection
    $\mathcal{C}\subset\mathcal{D}$ we either have that $|\limsup\mathcal{C}|\geqslant 1/2$ or
    $|\limsup(\mathcal{D}\setminus\mathcal{C})|\geqslant 1/2$, we obtain that
    $((h_I/|I|^{1/p},h_I/|I|^{1/p'})\colon I\in\mathcal{D})$ is strategically supporting in
    $L^p\times L^{p'}$.
  \end{proofcase}

  \begin{proofcase}[Case $X = L^1$]
    Corollary~6.3 in~\cite{MR4145794} asserts that $((h_I/|I|, h_I)\colon I\in\mathcal{D})$ has the
    (positive) factorisation property in $L^1\times L^\infty$.  By the theorem of Gamlen and
    Gaudet~\cite{gamlen:gaudet:1973}, we obtain that for any collection
    $\mathcal{C}\subset\mathcal{D}$ with $|\limsup\mathcal{C}|>0$, there exists a block basis
    $(b_I\colon I\in\mathcal{D})$ of $(h_I\colon I\in\mathcal{C})$ that is equivalent to
    $(h_I\colon I\in\mathcal{D})$ in both $L^1$ and $L^\infty$ (see
    also~\cite[p.~176~ff.]{mueller:2005} or the proof of~Theorem~6.1 in~\cite{MR4145794}) such that
    $(b_I\colon I\in\mathcal{D})$ also satisfies~\eqref{enu:dfn:strat-supp:iii}
    and~\eqref{enu:dfn:strat-supp:iv} in \Cref{dfn:strat-supp}.  Consequently,
    $((h_I/|I|,h_I)\colon I\in\mathcal{D})$ is strategically supporting in $L^1\times L^\infty$.
  \end{proofcase}

  \begin{proofcase}[Case $X=H^1$]
    By~\cite[Theorem~5.2]{MR4145794}, the Haar system is strategically reproducible (see
    also~\cite{mueller:1987,mueller:2005}).  Since the Haar system is an unconditional Schauder
    basis in $H^1$, it has the uniform diagonal factorisation property
    (see~\Cref{rem:uni-diag-fac-prop}).  Thus, by~\cite[Theorem~3.12]{MR4145794}, the system
    $((h_I/|I|, h_I)\colon I\in\mathcal{D})$ has the (positive) factorisation property in
    $H^1\times\bmo$.  By~\cite[Theorem~1(c)]{mueller:1987} and reviewing the proof, we find that (by
    the same mechanisms that were described in more detail above),
    $((h_I/|I|,h_I)\colon I\in\mathcal{D})$ is strategically supporting in $H^1\times\bmo$.
  \end{proofcase}
  
  \begin{proofcase}[Case $X=\mathrm{SL^\infty}$]
    We define $h_I^*\colon \mathrm{SL^\infty}\to\mathbb{R}$ by
    \begin{equation*}
      h_I^*\Bigl( \sum_{J\in\mathcal{D}} a_J h_J\Bigr)
      = a_I,
    \end{equation*}
    note that $h_I^*$ is linear and satisfies
    \begin{equation*}
      h_I^*(h_I) = 1,
      \qquad h_I^*(h_J) = 0
      \qquad\text{and}\qquad
      \|h_I^*\|_{(\mathrm{SL^\infty})^*} = 1,
    \end{equation*}
    for all $I\neq J\in\mathcal{D}$.  Let $Y$ denote the closed linear span of
    $\{h_I^*\colon I\in\mathcal{D}\}$ in $(\mathrm{SL^\infty})^*$.  We define a bilinear form
    $\langle\cdot, \cdot\rangle\colon \mathrm{SL^\infty}\times Y\to \mathbb{R}$ by
    \begin{equation*}
      \langle f, g\rangle
      = g(f),
      \qquad (f\in \mathrm{SL^\infty},\, g\in Y).
    \end{equation*}
    Given $f = \sum_{I\in\mathcal{D}} a_I h_I\in \mathrm{SL^\infty}$ and $\varepsilon > 0$, there
    exist $t\in [0,1]$ and $n\in\mathbb{N}$ such that
    \begin{equation}\label{eq:13}
      \biggl(\sum_{\substack{I\ni t\\|I|\geqslant 2^{-n}}} a_I^2\biggr)^{1/2}
      \geqslant (1-\varepsilon) \|f\|_{\mathrm{SL^\infty}}.
    \end{equation}
    Next, we define
    \begin{equation*}
      g
      = \sum_{\substack{I\ni t\\|I|\geqslant 2^{-n}}} a_I h_I^*
      \Big/\Bigl(\sum_{\substack{J\ni t\\|J|\geqslant 2^{-n}}} a_J^2\Bigr)^{1/2},
    \end{equation*}
    and note that $g\in\spn\{h_I^*\}\subset Y$.  The Cauchy--Schwarz inequality yields
    \begin{equation*}
      \Bigl|g\Bigl(\sum_{I\in\mathcal{D}} c_I h_I\Bigr)\Bigr|
      = \Bigl|\sum_{\substack{I\ni t\\|I|\geqslant 2^{-n}}} a_I c_I\Bigr|
      \Big/\Bigl(\sum_{\substack{J\ni t\\|J|\geqslant 2^{-n}}} a_J^2\Bigr)^{1/2}
      \leqslant \Bigl(\sum_{\substack{I\ni t\\|I|\geqslant 2^{-n}}} c_I^2\Bigr)^{1/2}
      \leqslant \Bigl\|\sum_{I\in\mathcal{D}} c_I h_I\Bigr\|_{\mathrm{SL^\infty}},
    \end{equation*}
    for all $\sum_{I\in\mathcal{D}} c_I h_I\in \mathrm{SL^\infty}$, thus,
    $\|g\|_{(\mathrm{SL^\infty})^*}\leqslant 1$.  Moreover, by the definition of $g$
    and~\eqref{eq:13}, we obtain
    \begin{equation*}
      \langle f, g\rangle
      = \biggl(\sum_{\substack{I\ni t\\|I|\geqslant 2^{-n}}} a_I^2\biggr)^{1/2}
      \geqslant (1-\varepsilon) \|f\|_{\mathrm{SL^\infty}}.
    \end{equation*}
    Altogether, we proved that~\eqref{eq:20} holds for $c=1$, \emph{i.e.},
    \begin{equation}\label{eq:30}
      \sup_{\|g\|_Y\leqslant 1} \langle f, g \rangle
      = \|f\|_{\mathrm{SL^\infty}}
      \qquad (f\in \mathrm{SL^\infty}).
    \end{equation}
    
    The second-named author proved in~\cite[Theorem~2.1]{lechner:2018:factor-SL} that the Haar
    system has the (positive) factorisation property in $\mathrm{SL^\infty}$.  Moreover, contained
    in the proof of~\cite[Theo\-rem~2.2]{lechner:2018:factor-SL} is the assertion that the Haar
    system is strategically supporting in $\mathrm{SL^\infty}$, which we will now
    elucidate.\smallskip

    Basically, we rerun the proof of~\cite[Theorem~2.2]{lechner:2018:factor-SL} with
    $T=I_{\mathrm{SL^\infty}}$.  First, we skip Step~1 (which amounts to $B_I = I$,
    $I\in\mathcal{D}$).  In Step~2, we split the dyadic intervals into two collections $\mathcal{E}$
    and $\mathcal{F}$ determined by the partition of $\mathbb{N}$ into $N_1$, $N_2$ as demanded by
    \Cref{dfn:strat-supp}.  The output of Step~2 are two normalised block bases
    $(\widetilde{h}_I\colon I\in\mathcal{D})$, $(\widetilde{h}_I^*\colon I\in\mathcal{D})$ which
    satisfy~\eqref{enu:dfn:strat-supp:i}--\eqref{enu:dfn:strat-supp:iv} in \Cref{dfn:strat-supp} and
    for which
    \begin{equation*}
      \text{either}\qquad
      \bigcup_I \widetilde{h}_I\subset\mathcal{E}
      \qquad\text{or}\qquad
      \bigcup_I \widetilde{h}_I\subset\mathcal{F}.
    \end{equation*}
    The equivalence between $(\widetilde{h}_I\colon I\in\mathcal{D})$ and
    $(h_I\colon I\in\mathcal{D})$ in $\mathrm{SL^\infty}$ is obtained directly from the boundedness
    of the operators $B,Q\colon \mathrm{SL^\infty}\to \mathrm{SL^\infty}$ given by
    \begin{equation*}
      Bf
      = \sum_{I\in\mathcal{D}}\langle f, h_I^*\rangle \widetilde{h}_I,
      \qquad
      Qf
      = \sum_{I\in\mathcal{D}}\langle f, \widetilde{h}_I^*\rangle h_I \qquad (f\in \mathrm{SL^\infty})
    \end{equation*}
    and the fact that $QB = I_{\mathrm{SL^\infty}}$.  The equivalence between
    $(\widetilde{h}_I^*\colon I\in\mathcal{D})$ and $(h_I^*\colon I\in\mathcal{D})$ in $Y$ is
    established by takin the adjoints of the operators (with respect to the dual pairing
    $(\mathrm{SL^\infty},Y,\langle\cdot,\cdot\rangle)$) and observing their boundedness (as maps
    from $Y$ to itself) using~\eqref{eq:30}.
  \end{proofcase}

  \begin{proofcase}[Case $X=\bmo$]
    The quickest way to describe why $((h_I,h_I/|I|)\colon I\in\mathcal{D})$ has the positive
    factorisation property and is strategically supporting in $\bmo\times H^1$, is to modify the
    above argument given for $\mathrm{SL^\infty}$.  We \emph{claim} that the same proof (beginning
    after~\eqref{eq:30}) and given for $\mathrm{SL^\infty}$ works if we replace $\mathrm{SL^\infty}$
    with $\bmo$.  Since Jones' compatibility conditions for $\mathrm{SL^\infty}$,
    (see~\cite[Section~3.1]{lechner:2018:factor-SL}) demand less than Jones' compatibility
    conditions for $\bmo$ (see~\cite{jones:1985}; see also~\cite[page~105 in
    Section~1.5]{mueller:2005}) the corresponding operators $B,Q$ (defined in the previous case) as
    mappings from $\bmo$ to itself might be unbounded.  However with some minor (yet important)
    modifications in constructing the block bases (which define the operators $B,Q$) we can also
    achieve Jones' compatibility conditions for $\bmo$, and thereby guarantee the boundedness of
    $B,Q\colon\bmo\to\bmo$.  We refer to~\cite{jones:1985,mueller:1987} and also
    to~\cite[Section~4.2]{mueller:2005}.\qedhere
  \end{proofcase}
\end{myproof}

\subsection{Application to \texorpdfstring{\pmb{$\ell^p(\ell^q)$}}{ellp(ellq)}}
\label{sec:application-ellpellq}

Let $\leqslant$ denote the linear order on $\mathbb{N}^2$ defined by the property that
$(i_0,j_0)\leqslant (i_1,j_1)$ whenever $i_0 < i_1$ or if $i_0=i_1$ and $j_0\leqslant j_1$.  For
$k\in\mathbb{N}$ let $(i(k),j(k))$ denote the $k$-th largest element in $\mathbb{N}^2$ with respect
to that linear order.  Vice versa, let $k(m,n)$ denote the unique $r\in\mathbb{N}$ such that
$(i(r),j(r)) = (m,n)$.

For $1\leqslant p,q\leqslant \infty$, $p'$ and $q'$ denote their respective H\"older conjugates,
i.e., $\frac{1}{p} + \frac{1}{p'} = 1$ and $\frac{1}{q} + \frac{1}{q'} = 1$.  The space
$\ell^p(\ell^q)$ is defined as
\begin{equation*}
  \bigl\{ (a_{ij}\colon i,j\in\mathbb{N})\colon \|(a_{ij})_{i,j}\|_{\ell^p(\ell^q)} < \infty\bigr\},
\end{equation*}
where the norm $\|\cdot\|_{\ell^p(\ell^q)}$ is given by
\begin{equation*}
  \|(a_{ij})_{i,j}\|_{\ell^p(\ell^q)}
  = \bigl\|(\|(a_{ij})_j\|_{\ell^q})_i\bigr\|_{\ell^p}
  = \biggl(\sum_{i=1}^\infty\Bigl(\sum_{j=1}^\infty |a_{ij}|^q\Bigr)^{p/q}\biggr)^{1/p}.
\end{equation*}
Let $(e_{ij}\colon i,j\in\mathbb{N})$ denote the standard Schauder basis in $\ell^p(\ell^q)$ and
$(f_{ij}\colon i,j\in\mathbb{N})$ the standard Schauder basis in $\ell^{p'}(\ell^{q'})$.

Next, we define a bilinear form
$\langle\cdot,\cdot\rangle\colon\ell^p(\ell^q)\times\ell^{p'}(\ell^{q'})\to\mathbb{R}$ by
\begin{equation*}
  \bigl\langle \sum_{i,j} a_{ij} e_{ij}, \sum_{k,l} b_{kl} f_{kl} \bigr\rangle = \sum_{k,l} a_{kl}
  b_{kl}\quad \Big(\sum_{i,j} a_{ij} e_{ij}\in\ell^p(\ell^q),\; \sum_{k,l} b_{kl} f_{kl}\in
  \ell^{p'}(\ell^{q'})\Big)
\end{equation*}
making $(\ell^p(\ell^q), \ell^{p'}(\ell^{q'}), \langle\cdot,\cdot\rangle)$ a dual pair.  One can
check that $(e_{ij}\colon i,j\in\mathbb{N})$ is a Schauder basis for $\ell^p(\ell^q)$ in the
topology $\sigma(\ell^p(\ell^q), \ell^{p'}(\ell^{q'}))$, and $(f_{ij}\colon i,j\in\mathbb{N})$ is a
basis for $\ell^{p'}(\ell^{q'})$ in the topology $\sigma(\ell^{p'}(\ell^{q'}), \ell^p(\ell^q))$.
Moreover, we note that~\eqref{eq:20} is satisfied with $c=1$, \emph{i.e.},
\begin{equation*}
  \sup_{\|y\|_{\ell^{p'}(\ell^{q'})}\leqslant 1} \langle x, y \rangle
  = \|x\|_{\ell^p(\ell^q)}
  \qquad (x\in \ell^p(\ell^q)).
\end{equation*}
The upcoming \Cref{thm:strat-rep-lplq,thm:strat-supp-lplq}, serve the purpose to verify the rest of
the hypotheses in \Cref{thm:max-ideal:2}: $((e_{ij}, f_{ij})\colon i,j\in\mathbb{N})$ has the
positive factorisation property and is strategically supporting.

\begin{thm}\label{thm:strat-rep-lplq}
  Suppose that $1\leqslant p, p^\prime\leqslant\infty$, $1 < q, q^\prime < \infty$ and
  $\frac{1}{p} + \frac{1}{p'} = 1$, $\frac{1}{q} + \frac{1}{q'} = 1$.  Then
  $((e_{ij}, f_{ij})\colon i,j\in\mathbb{N})$ has the positive factorisation property in
  $\ell^p(\ell^q)\times\ell^{p'}(\ell^{q'})$.
\end{thm}

\begin{proof}
  It is straightforward to show that the standard Schauder basis $(e_n)_{n=1}^\infty$ of $\ell^q$ is
  $1$-strategically reproducible, whenever $1 < q < \infty$; we therefore omit the proof.  Since
  $(e_n)_{n=1}^\infty$ is $1$-unconditional, \Cref{rem:uni-diag-fac-prop} asserts that
  $(e_n)_{n=1}^\infty$ has the uniform diagonal factorisation property; to be more precise, one can
  quickly check that $(e_n)_{n=1}^\infty$ has in fact the $1/\delta$-diagonal factorisation
  property, $\delta > 0$.  Thus, \cite[Theorem~7.6]{MR4145794} yields that
  $((e_{ij}, f_{ij})\colon i,j\in\mathbb{N})$ has the $1/\delta$-factorisation property in
  $\ell^p(\ell^q)\times\ell^{p'}(\ell^{q'})$ for all $1\leqslant p < \infty$, $1 < q < \infty$.  In
  particular, by \Cref{rem:factorisation-property} we obtain that
  $((e_{ij}, f_{ij})\colon i,j\in\mathbb{N})$ has the positive factorisation property in
  $\ell^p(\ell^q)\times\ell^{p'}(\ell^{q'})$ for $1\leqslant p < \infty$, $1 < q < \infty$.
  Moreover, by~\cite[Corollary~3.8]{lechner:motakis:mueller:schlumprecht:2019},
  $((e_{ij}, f_{ij})\colon i,j\in\mathbb{N})$ has the $1/\delta$-factorisation property in
  $\ell^\infty(\ell^q)\times\ell^1(\ell^{q'})$ for all $1 < q < \infty$.  Invoking
  \Cref{rem:factorisation-property} yields that $((e_{ij}, f_{ij})\colon i,j\in\mathbb{N})$ has the
  positive factorisation property in $\ell^\infty(\ell^q)\times\ell^1(\ell^{q'})$, $1 < q < \infty$.
\end{proof}

\begin{thm}\label{thm:strat-supp-lplq}
  Let $1\leqslant p,q\leqslant\infty$ and $\frac{1}{p} + \frac{1}{p'} = 1$,
  $\frac{1}{q} + \frac{1}{q'} = 1$.  Then $((e_{ij}, f_{ij})\colon i,j\in\mathbb{N})$ is
  strategically supporting in $\ell^p(\ell^q)\times\ell^{p'}(\ell^{q'})$.
\end{thm}

\begin{proof}
  Let $N_1,N_2$ be a partition of $\mathbb{N}$.  For each $i\in\mathbb{N}$, we define the sets
  \begin{equation*}
    \mathcal{A}_i^m = \{j\in\mathbb{N}\colon k(i,j)\in N_m \},
    \qquad m=1,2.
  \end{equation*}
  Observe that for each $i\in\mathbb{N}$, either of the two sets $\mathcal{A}_i^m$, $m=1,2$ is
  infinite.  Next, we define $\mathcal{B}^m = \{i\in\mathbb{N}\colon |\mathcal{A}_i^m| = \infty\}$,
  $m=1,2$ and note that either $\mathcal{B}^1$ or $\mathcal{B}^2$ must be infinite.  Pick
  $m_0\in\{1,2\}$ such that $\mathcal{B}^{m_0}$ is infinite.  When properly relabelled,
  $(e_{ij}\colon i\in\mathcal{B}^{m_0},\ j\in\mathcal{A}_i^{m_0})$ is then equivalent to
  $(e_{ij}\colon i,j\in\mathbb{N})$ in $\ell^p(\ell^q)$ and
  $(f_{ij}\colon i\in\mathcal{B}^{m_0},\ j\in\mathcal{A}_i^{m_0})$ is equivalent to
  $(f_{ij}\colon i,j\in\mathbb{N})$ in $\ell^{p'}(\ell^{q'})$.
\end{proof}

Combining \Cref{thm:strat-rep-lplq} with \Cref{thm:strat-supp-lplq}, the remarks at the beginning of
\Cref{sec:application-ellpellq} and \Cref{thm:max-ideal:2} yield the following conclusion that had
been known to the experts, yet to the best of our knowledge it has not been recorded in the
literature.  Indeed, primarity of various infinite direct sums of Banach spaces has been assessed in
the literature (see, \emph{e.g.}, \cite{capon:1982:1}, \cite{casazza:kottman:lin:1977},
\cite{MR507272}, and more recently in~\cite{MR3977084}), however not in relation to ideals of
operators acting thereon.  A closer inspection of the proofs often allows one to deduce closedness
under addition of the set defined in \eqref{eq:Mx}. \Cref{cor:max-ideals-lplq} puts these results on
a more systematic footing.
\begin{cor}\label{cor:max-ideals-lplq}
  Let $1\leqslant p\leqslant\infty$ and $1 < q < \infty$.  Then $\mathscr{M}_{\ell^p(\ell^q)}$ is
  the unique closed proper maximal ideal of $\mathscr{B}(\ell^p(\ell^q))$.
\end{cor}

\subsection{Application to mixed-norm Lebesgue spaces}
\label{sec:appl-mixed-norm}

The collection of dyadic rectangles $\mathcal{R}$ is given by
\begin{equation*}
  \mathcal{R}
  = \{ I\times J\colon I,J\in\mathcal{D}\}.
\end{equation*}
The biparameter Haar system $(h_{I,J}\colon I,J\in\mathcal{D})$ is given by
\begin{equation*}
  h_{I,J}(x,y) = h_I(x)h_J(y)\quad (x,y\in [0,1], I,J\in\mathcal{D}).
\end{equation*}
For $1\leqslant p,q < \infty$, the mixed-norm Lebesgue space $L^p(L^q)$ is the completion of
\begin{equation*}
  \mathcal{H}
  := \spn\bigl\{ h_{I,J}\colon I,J\in\mathcal{D} \bigr\}
\end{equation*}
under the norm $\|\cdot\|_{L^p(L^q)}$ given by
\begin{equation*}
  \|f\|_{L^p(L^q)}
  = \biggl( \int_0^1 \Bigl( \int_0^1 |f(x,y)|^q \mathrm{d}y \Bigr)^{p/q} \mathrm{d}x \biggr)^{1/p}.
\end{equation*}
The mixed-norm dyadic Hardy space $H^p(H^q)$ is defined as the completion of $\mathcal{H}$ under the
square function norm $\|\cdot\|_{H^p(H^q)}$ given by
\begin{equation*}
  \Bigl\| \sum_{I,J\in\mathcal{D}} a_{I,J} h_{I,J}\Bigr\|_{H^p(H^q)}
  = \Bigl\| \Bigl(\sum_{I,J\in\mathcal{D}} a_{I,J}^2 h_{I,J}^2\Bigr)^{1/2} \Bigr\|_{L^p(L^q)}.
\end{equation*}

M.~Capon~\cite[Proposition~I.1]{capon:1982:2} shows that the identity operator is an isomorphism
between $L^p(L^q)$ and $H^p(H^q)$, whenever $1 < p,q < \infty$ and that the biparameter Haar system
$(h_{I,J}\colon I,J\in\mathcal{D})$ is an unconditional Schauder basis for $L^p(L^q)$,
$1 < p,q < \infty$.

The following crucial theorem for $L^p(L^q)$, $1 < p,q < \infty$ is due to
M.~Capon~\cite[Theorem~1.3]{capon:1982:2}, and for $H^1(H^1)$ it is due to
P.F.X.~Müller~\cite[Theorem~4]{mueller:1994}.
\begin{thm}\label{thm:capon-mueller}
  Let $X$ denote either of the spaces $L^p(L^q)$, $1 < p,q < \infty$ or $H^1(H^1)$.  Given
  $\mathcal{C}\subset\mathcal{R}$, $I\in\mathcal{D}$ and $t\in[0,1]$, we define
  \begin{equation*}
    \mathcal{C}_I
    = \{J\in\mathcal{D}\colon I\times J\in\mathcal{C}\}
    \quad\text{and}\quad
    C_t
    = \limsup\{ I \in\mathcal{D}\colon t\in\limsup \mathcal{C}_I\}.
  \end{equation*}
  If $|\{t\in [0,1]\colon |C_t| \geqslant 1/2\}| > 0$, then one may find a block basis
  $(\widetilde{h}_{I,J}\colon I\times J\in\mathcal{D})$ of $(h_{I,J}\colon I\times J\in\mathcal{C})$
  which is equivalent to $(h_{I,J}\colon I,J\in\mathcal{D})$ in $X$ and in $X^*$.
\end{thm}
The result for $L^p(L^q)$ is due to M.~Capon~\cite{capon:1982:2} (note that the dual of $L^p(L^q)$
is $L^{p'}(L^{q'})$, where $1/p + 1/p' = 1$ and $1/q + 1/q' = 1$), the result for $H^1(H^1)$ is due
to P.F.X.~Müller~\cite{mueller:1994} (dualising the natural projection onto the block basis yields
the equivalence of the block basis in $(H^1(H^1))^*$; also see~\Cref{rem:basic-operators}).  We
refer to~\cite{2021arXiv210210088L,lechner:mueller:2015,mueller:1994,mueller:2005} for related
works.

\begin{cor}\label{cor:capon-mueller}
  Let $1 < p,p',q,q' < \infty$ with $1/p + 1/p' = 1$ and $1/q + 1/q' = 1$.  Then the following
  systems are strategically supporting:
  \begin{romanenumerate}
  \item $((h_{I,J}/(|I|^{1/p}|J|^{1/q}), h_{I,J}/(|I|^{1/p'}|J|^{1/q'}))\colon I,J\in\mathcal{D})$
    in $L^p(L^q)\times L^{p'}(L^{q'})$,
  \item $((h_{I,J}/(|I||J|), h_{I,J})\colon I,J\in\mathcal{D})$ in $H^1(H^1)\times (H^1(H^1))^*$.
  \end{romanenumerate}
\end{cor}

\begin{proof}[Proof of \Cref{cor:capon-mueller}]
  Additionally to the definitions $\mathcal{C}_I$ and $C_t$ as in \Cref{thm:capon-mueller}, we
  define
  \begin{equation*}
    \mathcal{C}_I'
    = \{J\in\mathcal{D}\colon I\times J\in\mathcal{R}\setminus\mathcal{C}\}
    \quad\text{and}\quad
    C_t'
    = \limsup\{ I \in\mathcal{D}\colon t\in\limsup \mathcal{C}_I'\}.
  \end{equation*}
  Certainly, for each $t\in [0,1]$ and $I\in\mathcal{D}$, either $t\in\limsup\mathcal{C}_I$ or
  $t\in\limsup\mathcal{C}_I'$, and thus $C_t\cup C_t' = [0,1]$.  Consequently, either
  \begin{equation*}
    |\{t\in [0,1]\colon |C_t| \geqslant 1/2\}| \geqslant 1/2
    \qquad\text{or}\qquad
    |\{t\in [0,1]\colon |C_t'| \geqslant 1/2\}| \geqslant 1/2.
  \end{equation*}
  Applying \Cref{thm:capon-mueller}, yields a block basis of either
  $(h_{I,J}\colon I\times J\in\mathcal{C})$ or
  $(h_{I,J}\colon I\times J\in\mathcal{R}\setminus\mathcal{C})$ which is equivalent to
  $(h_{I,J}\colon I\times J\in\mathcal{D})$.  Thus, we showed
  \Cref{dfn:strat-supp}~\eqref{enu:dfn:strat-supp:i} and~\eqref{enu:dfn:strat-supp:ii} are
  satisfied.  The other two properties of \Cref{dfn:strat-supp} are obvious.
\end{proof}

\begin{thm}\label{thm:HpHq-factor-prop}
  Let $1 < p,q < \infty$.  Then the following systems have the positive factorisation property:
  \begin{romanenumerate}
  \item $((h_{I,J}/(|I|^{1/p}|J|^{1/q}), h_{I,J}/(|I|^{1/p'}|J|^{1/q'}))\colon I,J\in\mathcal{D})$
    in $L^p(L^q)\times L^{p'}(L^{q'})$,
  \item $((h_{I,J}/(|I||J|), h_{I,J})\colon I,J\in\mathcal{D})$ in $H^1(H^1)\times (H^1(H^1))^*$.
  \end{romanenumerate}
\end{thm}

\begin{proof}
  For $1\leqslant p,q < \infty$, \cite[Theorem~3.1]{laustsen:lechner:mueller:2015} yields that the
  identity operator factors through every operator $T\colon H^p(H^q)\to H^p(H^q)$ satisfying
  \begin{equation*}
    \inf_{I,J}\frac{\langle T h_{I,J}, h_{I,J}\rangle}{|I||J|} > 0.
  \end{equation*}
  By~\cite[Proposition~I.1]{capon:1982:2}, the identity operator is an isomorphism between
  $L^p(L^q)$ and $H^p(H^q)$, whenever $1 < p,q < \infty$.
\end{proof}

Using \Cref{cor:capon-mueller} together with \Cref{thm:HpHq-factor-prop} and subsequently applying
\Cref{thm:max-ideal:2} yields \Cref{cor:max-ideals-HpHq}, below.
\begin{cor}\label{cor:max-ideals-HpHq}
  Let $X$ denote either of the spaces $H^1(H^1)$ or $L^p(L^q)$, $1 < p,q < \infty$.  Then
  $\mathscr{M}_X$ is the unique closed proper maximal ideal of $\mathscr{B}(X)$.
\end{cor}
The results in \Cref{cor:max-ideals-HpHq} also follow by
combining~\cite[Section~5]{dosev:johnson:2010} with~\cite{mueller:1994} for $H^1(H^1)$ and
with~\cite{capon:1982:2} for $L^p(L^q)$.

\bibliographystyle{abbrv}%
\bibliography{bibliography}%

\end{document}